\documentclass[12pt]{amsart}
\newcommand{\C}{\mathbb{C}}
\newcommand{\Pro}{\mathcal{P}}
\newcommand{\End}{\operatorname{End}}
\newcommand{\Hom}{\operatorname{Hom}}
\newcommand{\Ext}{\operatorname{Ext}}
\newcommand{\Coh}{\operatorname{Coh}}
\newcommand{\Str}{\mathcal{O}}
\newcommand{\sgn}{\operatorname{sgn}}

\newcommand{\Loc}{\operatorname{Loc}}

\newcommand{\Z}{\mathbb{Z}}
\newcommand{\Weyl}{\mathsf{W}}
\newcommand{\A}{\mathcal{A}}
\newcommand{\F}{\mathbb{F}}
\newcommand{\D}{\mathcal{D}}
\newcommand{\gr}{\operatorname{gr}}
\newcommand{\g}{\mathfrak{g}}
\newcommand{\Fr}{\operatorname{Fr}}

\newcommand{\B}{\mathcal{B}}

\newcommand{\Spec}{\operatorname{Spec}}
\newcommand{\Q}{\mathbb{Q}}
\newcommand{\Ring}{\mathsf{R}}
\newcommand{\E}{\mathcal{E}}
\newcommand{\M}{\mathcal{M}}
\usepackage{amssymb}
\usepackage{amsmath}
\newtheorem{Thm}{Theorem}[section]
\newtheorem{Prop}[Thm]{Proposition}
\newtheorem{Cor}[Thm]{Corollary}
\newtheorem{Lem}[Thm]{Lemma}
\theoremstyle{definition}
\newtheorem{Ex}[Thm]{Example}

\newtheorem{Rem}[Thm]{Remark}

\numberwithin{equation}{section}
\author{Ivan Losev}
\title{On inductive construction of Procesi bundles}
\address{Department
of Mathematics, University of Toronto, ON,  Canada \&
National Research University Higher School of Economics, Moscow, Russian Federation}
\email{ivan.loseu@gmail.com}
\begin{document}
\begin{abstract}
A Procesi bundle, a rank $n!$ vector bundle on the Hilbert scheme $H_n$ of $n$ points in $\mathbb{C}^2$,
was first constructed by Mark Haiman in his proof of the $n!$ theorem by using a complicated combinatorial argument. 
Since then alternative constructions
of this bundle were given by Bezrukavnikov-Kaledin and by Ginzburg. In this paper we give a geometric/
representation-theoretic proof of the inductive formula for the Procesi bundle that plays an important role
in Haiman's construction. Then we use the inductive formula 
to prove a weaker version of the $n!$ theorem: the normalization
of Haiman's isospectral Hilbert scheme is Cohen-Macaulay and Gorenstein, and the normalization morphism
is bijective. This improves an earlier result of Ginzburg.
\end{abstract}
\maketitle

\section{Introduction}
\subsection{$n!$ theorem}
The Hilbert scheme $H_n:=\operatorname{Hilb}_n(\C^2)$ parameterizes the codimension $n$ ideals in $\C[x,y]$.
It is known to be a smooth irreducible algebraic variety of dimension $2n$. It comes with the
Hilbert-Chow morphism  $\rho_n:H_n\rightarrow V_n/S_n$, where we write $V_n$ for $(\C^2)^{\oplus n}$
and the symmetric group $S_n$ acts by permuting the $n$ copies of $\C^2$. The morphism sends a point
in $H_n$ --  a codimension $n$ ideal in $\C[x,y]$ -- to its support counted with multiplicities.
This morphism is known to be a resolution of singularities.

Following Haiman, we define the isospectral Hilbert scheme $IH_n$ as the fiber product
$V_n\times_{V_n/S_n}H_n$ with reduced scheme structure. By definition, there is
a finite morphism $\eta:IH_n\rightarrow H_n$. It is the quotient morphism
for the natural action of $S_n$ on $IH_n$.

Here is a geometric version Haiman's $n!$ theorem, \cite[Theorem 3.1]{Haiman_n!}, (there is also an elementary version
that has to do with the spaces of partial derivatives of certain two-variable generalizations
of the Vandermonde determinant, that version follows from the geometric one).

\begin{Thm}\label{Thm:n!_geometric}
The variety $IH_n$ is Cohen-Macaulay and Gorenstein.
\end{Thm}

Note that, in order for $IH_n$ to be Cohen-Macaulay, we do need to consider $V_n\times_{V_n/S_n}H_n$
with its reduced scheme structure: the fiber product with the natural scheme structure can be shown to
be non-reduced while generically reduced. This can never happen for Cohen-Macaulay schemes.

Let us explain a motivation behind Theorem \ref{Thm:n!_geometric}.
Note that the two-dimensional torus $T:=(\C^\times)^2$ acts on $\C^2$ in a natural way:
$(t_1,t_2).(x,y)=(t_1^{-1}x, t_2^{-1}y)$. The action induces an action on $V_n$, on $\C[x,y]$
by algebra automorphisms and hence an action on $H_n$. The Hilbert-Chow map $H_n\rightarrow V_n/S_n$
is $T$-equivariant so we get a $T$-action on $IH_n$ as well. The $T$-fixed
points in $H_n$ precisely correspond to the monomial codimension $n$ ideals in $\C[x,y]$.
The latter are naturally labelled by Young diagrams with $n$ boxes: the diagram corresponding to
a monomial ideal indicates which monomials do not lie in the ideal. Let us write $x_\lambda$ for the
fixed point labelled by the Young diagram $\lambda$.

Now consider the sheaf $\Pro^H_n:=\eta_* \Str_{IH_n}$ (the superscript ``H''
is for Haiman). Since $IH_n$ is Cohen-Macaulay
and $\eta$ is finite, $\eta$ is flat. The degree is easily seen to be $n!$. So $\Pro^H_n$
is a degree $n!$ vector bundle. The idea to look for such a bundle was due to Procesi,
so Haiman called $\Pro^H_n$ the {\it Procesi bundle}. The bundle $\Pro^H_n$ is $T$-equivariant.
In particular, its fibers at the fixed points, $\Pro^H_{\lambda}:=(\Pro^H_n)_{x_\lambda}$
are bigraded $S_n$-modules. Another hard theorem from \cite{Haiman_n!} says that the Frobenius
character of the bigraded $S_n$-module $(\Pro^H_n)_{x_\lambda}$ is the modified Macdonald
polynomial $\tilde{H}_\lambda(q,t)$. This, in particular, proves the famous Macdonald
positivity conjecture that says that the Macdonald polynomials are Schur positive.

\subsection{Inductive construction}
Let us say a few words about Haiman's proof of Theorem \ref{Thm:n!_geometric},
see \cite[Sections 3,4]{Haiman_n!} and also \cite[Sections 5.6]{Haiman_CDM}
for an overview.
The construction is inductive in nature and utilizes the {\it nested Hilbert
scheme} $H_{n,n-1}$. This scheme parameterizes pairs of ideals $J\subset J'$
in $\C[x,y]$ such that $\operatorname{codim}_{\C[x,y]} J=n$ and $\operatorname{codim}_{\C[x,y]}J'=n-1$.
It turns out, \cite{Cheah}, that $H_{n,n-1}$ is an irreducible smooth variety of dimension
$2n$. It comes with two morphisms $\beta: H_{n,n-1}\rightarrow H_{n-1}\times \C^2$
sending $(J'\subset J)$ to $(J',\operatorname{Supp}(J'/J))$ and $\alpha:H_{n,n-1}\rightarrow H_n$ forgetting $J'$.

Set $IH_{n,n-1}:=H_{n,n-1}\times_{H_{n-1}}IH_{n-1}$. This scheme comes with a natural
morphism to $IH_{n}$, denote it by $\hat{\alpha}$, it is induced by $\alpha$. Haiman deduced
Theorem \ref{Thm:n!_geometric} from the following equality
\begin{equation}\label{eq:Haiman_iso}
\hat{\alpha}_* \Str_{IH_{n,n-1}}=\Str_{IH_n}.
\end{equation}
Here and below we will use the notation like $\hat{\alpha}_*$ for the {\it derived} functor.

The proof of (\ref{eq:Haiman_iso}) is based on Haiman's {\it polygraph} theorem,
\cite[Theorem 4.1]{Haiman_n!}, which
is proved in \cite{Haiman_n!} by  explicit combinatorial/ commutative-algebraic computations.
Since these computations are extremely complicated, there were several attempts
to under various parts of Haiman's construction more conceptually. The present paper
contributes to this goal.

\subsection{Related developments}
Before  explaining results of this paper we want to discuss some subsequent developments
related to Haiman's work starting with a construction that  ours is based upon: a construction
of the Procesi bundle due to Bezrukavnikov and Kaledin, \cite{BK}.

An important property of the Procesi bundle $\Pro^H_n$ observed by Haiman in \cite{Haiman_vanishing}
is  the derived McKay equivalence $$R\operatorname{Hom}_{\Str_{H_n}}(\Pro^H_n,\bullet):
D^b(\Coh H_n)\xrightarrow{\sim} D^b(\C[\underline{x},\underline{y}]\operatorname{-mod}^{S_n}),$$
where in the target we write $\underline{x}$ for $x_1,\ldots,x_n$, $\underline{y}$ for $y_1,\ldots,y_n$,
and the superscript ``$S_n$'' means that we consider the category of $S_n$-equivariant modules.
This equivalence sends $\Pro^H_n$ to the smash-product algebra $\C[\underline{x},\underline{y}]\# S_n$
(whose category of modules is $\C[\underline{x},\underline{y}]\operatorname{-mod}^{S_n}$). A consequence
(and, in fact, an equivalent condition thanks to \cite[Proposition 2.2]{BK}) is that
\begin{equation}\label{eq:REnd}
R\operatorname{End}(\Pro^H_n)=\C[\underline{x},\underline{y}]\# S_n.
\end{equation}

On the other hand, derived equivalences of this sort are classical in geometric Representation
theory, where they come from the derived Beilinson-Bernstein type localization theorem.
Inspired by the version of this theorem in positive characteristic proved in \cite{BMR}, in \cite{BK}
Bezrukavnikov and Kaledin established a derived McKay equivalence and constructed
an analog of a Procesi bundle on any symplectic resolution of any symplectic quotient singularity,
a basic example is the resolution $H_n$ of $V_n/S_n$. The latter variety admits a standard
quantization, $\Weyl(V_n)^{S_n}$, where we write $\Weyl(V_n)$ for the Weyl algebra of
the symplectic vector space $V_n$. Over an algebraically closed field $\F$ of positive characteristic, the algebra $\Weyl(V_n)$
is an Azumaya algebra over the Frobenius twist $V_n^{(1)}$. The restrictions of $\Weyl(V_n)^{S_n}$
and $\F[V_n^{(1)}]\# S_n$ to the formal neighborghood of $0$ in $V_n^{(1)}$ are Morita equivalent
provided the characteristic is large enough.
Further, one can find a ``Frobenius constant'' quantization
$\underline{\D}$ of $H_n$ whose algebra of (derived) global sections (automatically a quantization
of $V_n/S_n$) is $\Weyl(V_n)^{S_n}$. Being Frobenius constant means, in particular,
that $\underline{\D}$ can be viewed as an Azumaya algebra on $H_n^{(1)}$. The restriction of $\underline{\D}$
to the formal neighborhood of the zero fiber of $H_n^{(1)}\rightarrow V_n^{(1)}/S_n$
splits. From the indecomposable summands of a splitting bundle one can form
an analog of a Procesi bundle. Then thanks to the rigidity (the absence of higher
self-extensions) one can first extend this analog of a Procesi bundle to the whole variety $H_n^{(1)}$ and then
lift to characteristic $0$. One gets a vector bundle $\Pro_n^{BK}$ on $H_n$
satisfying (\ref{eq:REnd}).

The bundle $\Pro_n^{BK}$ can be shown to  coincide with $\Pro_n^H$ thanks
to results of \cite{L_Procesi}, where ``abstract'' Procesi bundles were classified:
it was shown that there are exactly two abstract {\it normalized} Procesi bundles that are dual
to one another. In this paper we reprove Theorem \ref{Thm:n!_geometric} using the bundle
$\Pro_n^{BK}$. Let us also point out that the Macdonald positivity was reproved
(and generalized)  in \cite{BF} by using the construction of $\Pro_n^{BK}$ from
\cite{BK}.

We would also like to mention several other related developments, although they do not
play any role in our proofs. Ginzburg proved that the normalization of $IH_n$
is Cohen-Macaulay and Gorenstein using the Hodge filtration on the  Hotta-Kashiwara
D-module, see \cite[Proposition 8.2.4]{Ginzburg}.
The variety $IH_n$ was shown to be normal by Haiman but presently there is
no independent proof of the normality. Gordon, \cite{Gordon_Procesi}, used Ginzburg's
construction to deduce the Macdonald positivity. There are also other proofs of
the Macdonald positivity, e.g.,  \cite{GrHa}.

\subsection{Results and ideas of proof}
We now proceed to explaining the ideas of our proof.
First, let us record a consequence of (\ref{eq:Haiman_iso}):
\begin{equation}\label{eq:Haiman_iso_Pro}
\alpha_*(\beta^* \left(\Pro^H_{n-1}\boxtimes \Str_{\C^2}\right))=\Pro^H_{n}.
\end{equation}

Our goal is to establish an analog of this equality for the Procesi bundle
constructed by Bezrukavnikov and Kaledin:

\begin{Thm}\label{Thm:Procesi_inductive}
We have a $\C[\underline{x},\underline{y}]\#S_{n-1}$-linear isomorphism
\begin{equation}\label{eq:BK_iso_Pro}
\alpha_*(\beta^* \left(\Pro^{BK}_{n-1}\boxtimes \Str_{\C^2}\right))=\Pro^{BK}_n,
\end{equation}
\end{Thm}

We will use this theorem to strengthen a result of Ginzburg,
\cite[Proposition 8.2.4]{Ginzburg}.

\begin{Thm}\label{Thm:bijective_normalization}
The normalization of $IH_n$ is Cohen-Macaulay and Gorenstein. Moreover
the normalization morphism for $IH_n$ is bijective.
\end{Thm}


In the remainder of this section we will discuss the ideas of our proof of
(\ref{eq:BK_iso_Pro}).

We will prove a  ``quantum analog'' of (\ref{eq:BK_iso_Pro}) in characteristic
$p\gg 0$. Namely, note that $H_{n,n-1}$ is a lagrangian subvariety in $H_n\times (H_{n-1}\times \C^2)$.
We will show that, over $\C$, the sheaf $\Str_{H_{n,n-1}}$ can be quantized to a filtered $\mathcal{D}_{n-1,\diamondsuit}$-$\mathcal{D}_{n}$-bimodule $\mathcal{D}_{n,n-1}$ (where $\mathcal{D}_n$ is the characteristic $0$ version of the quantization of $H_n$ used in
\cite{BK} to produce the Procesi bundle and $\mathcal{D}_{n-1,\diamondsuit}$ is an analogous quantization of
$H_{n-1}\times \C^2$). We will show that $\Gamma(\mathcal{D}_{n,n-1})=\Weyl(V_n)^{S_{n-1}}$, while the
higher cohomologies of $\mathcal{D}_{n,n-1}$ vanish.

Then we will show that the quantization $\D_{n,n-1}$ can be reduced mod $p$ for $p\gg 0$.
Let us denote corresponding quantization by $\D^\F_{n,n-1}$. We still have
$R\Gamma(\mathcal{D}_{n,n-1}^\F)=\Weyl(V_n^\F)^{S_{n-1}}$. Also we will show  that
$\D^\F_{n,n-1}$ gives rise to a splitting bundle for the restriction to $H_{n,n-1}^{\F,(1)}$ of the
Azumaya algebra arising from $\mathcal{D}_n^{\F,opp}\otimes \mathcal{D}_{n-1,\diamondsuit}^{\F}$. Using this splitting
bundle and the Bezrukavnikov-Kaledin construction of the Procesi bundle (and  some additional work) we deduce
(\ref{eq:BK_iso_Pro}) over $\F$ from here. Using this, we will prove Theorems
\ref{Thm:Procesi_inductive} and \ref{Thm:bijective_normalization} over $\C$.

\subsection{Conventions and notation}
In this paper we work with various base fields and rings. They include $\C$, $\F:=\overline{\F}_p$
for $p$ sufficiently large, a large finite localization $\Ring$ of $\Z$ and some others.
The base ring is indicated as a supercript, unless it is $\C$, in which case we skip
the superscript $\C$.

All pullback and pushforward functors are derived, while $\Gamma$ denotes the usual
non-derived global section functor.

{\it Notation for varieties and algebras}. Throughout the paper we use the following notation
for varieties and (commutative and noncommutative) algebras:
\begin{itemize}
\item $V_n:=(\C^2)^{\oplus n}$. By $x_1,y_1,\ldots, x_n,y_n$ we denote the natural basis in $V_n$
so that $\C[V_n]=\C[x_1,\ldots,x_n,y_1,\ldots,y_n]$.
\item $V_n^0$ is the regular locus in $V_n$ consisting of $n$-tuples of pairwise distinct points.
\item $V_n^1$ is the open locus in $V_n$ consisting of all points whose stabilizer in $S_n$
is either trivial or is generated by a single transposition.
\item $H_n$ is the Hilbert scheme of $n$ points in the plane.
\item $H_{n,n-1}$  is the nested Hilbert scheme parameterizing inclusions $J'\subset J$
with $J\in H_n, J'\in H_{n-1}$.
\item $IH_n$ is the isospectral Hilbert scheme, i.e., $V_n\times_{V_n/S_n}H_n$ with reduced
scheme structure.
\item $U_n$ is the universal degree $n$ family over $H_n$.
\item $\Weyl(V_n)$ is the Weyl algebra of the symplectic vector space $V_n$.
\item $\A_n:=\Weyl(V_n)^{S_n}$.
\item $\widetilde{A}_n:=\C[V_n]\# S_n$.
\item $X_n$ is the relative spectrum of the Procesi bundle on $H_n$, constructed in
Section \ref{S_n!_proof}.
\item $X_{n,n-1}:=X_{n-1}\times_{H_{n-1}}H_{n,n-1}$.
\end{itemize}

We also often use the subscript $\diamondsuit$. The meaning is the multiplication by
$\C^2$ for varieties or by $\C[V_2]$ or $\Weyl(V_2)$ for the algebras above.
For example, $V^0_{n-1,\diamondsuit}$ stands for $V^0_{n-1}\times \C^2$ (where the second
factor corresponds to the last two coordinates) and  $\A_{n-1,\diamondsuit}$ stands for
$\A_{n-1}\otimes \Weyl(V_2)$.

As usual, the superscript ``$^{(1)}$'' indicates the Frobenius twist.

{\it Notation for morphisms}:
\begin{itemize}
\item $\alpha_n$ is the natural morphism $H_{n,n-1}\rightarrow H_n$.
\item $\bar{\alpha}_n$ is the natural morphism $H_{n,n-1}\rightarrow U_n$.
\item $\widetilde{\alpha}_n$ is the natural morphism $X_{n,n-1}\rightarrow X_n$ defined in
Section \ref{S_n!_proof}.
\item $\beta_n$ is the natural morphism $H_{n,n-1}\rightarrow H_{n-1,\diamondsuit}$.
\item $\widetilde{\beta}_n$ is the natural morphism $X_{n,n-1}\rightarrow X_{n-1,\diamondsuit}$.
\item $\eta_n$ is the natural morphism $IH_n\rightarrow H_n$.
\item $\mu$ stands for a moment  map.
\item $\pi_n$ is the natural morphism $X_n\rightarrow H_n$  introduced in Section
\ref{S_n!_proof}. Similarly, $\pi_{n,n-1}:X_{n,n-1}\rightarrow H_{n,n-1}$ is the natural
morphism.
\item $\rho_n:H_n\rightarrow V_n/S_{n}$ is the Hilbert-Chow morphism. Similarly, we write
$\widetilde{\rho}_n,\widetilde{\rho}_{n,n-1},\bar{\rho}_n$ for the natural morphisms
$X_n\rightarrow V_n, X_{n,n-1}\rightarrow V_n, U_n\rightarrow V_n/S_{n-1}$, respectively.
\item $\tau_n$ is also the natural morphism $U_n\rightarrow H_n$.
\end{itemize}

Finally, let us list our notation for various sheaves.

\begin{itemize}
\item $\D_n$ denotes a certain microlocal quantization of $H_n$ recalled in Section \ref{SS_quant_general}.
\item $\underline{\D}^\F_n$ denotes a Frobenius constant quantization of $H_n^\F$ obtained from $\D^\F_n$.
\item $\D_{n,n-1}$ denotes a certain microlocal quantization of the lagrangian subvariety $H_{n,n-1}
\subset H_n\times H_{n-1,\diamondsuit}$ constructed in Section \ref{S_quant_nested},
and $\underline{\D}_{n,n-1}^\F$ is the ``Frobenius constant'' version.
\item $K_Y$ denotes the canonical bundle of a smooth variety $Y$.
\item $\Str_n,\Str_{n,n-1}$ denote the structure sheaves of $H_n,H_{n,n-1}$. More generally,
we write $\Str_n(k)$ for the line bundle on $H_n$ corresponding to $k\in \Z$ and
$\Str_{n,n-1}(k,\ell)$ for $\beta_n^*\Str_{n-1}(k)\otimes \alpha_n^* \Str_n(\ell)$, a line bundle on
$H_{n,n-1}$.
\item $\Pro_n$ is the Procesi bundle on $H_n$ whose construction is recalled in
Section \ref{SS_Procesi_constr}.
\item $\mathcal{T}_n:=\tau_{n*}\Str_{U_n}$ is the tautological bundle on $H_n$.
\end{itemize}

{\bf Acknowledgements}: I would like to thank Roman Bezrukavnikov, Pavel Etingof,
Victor Ginzburg, Evgeny Gorsky, Andrei Negut, Alexei Oblomkov, and Lev Rozansky
for stimulating discussions. I would also like to thank Roman Bezrukavnikov for the discussion that 
allowed me to find a fatal mistake in a previous version of this paper that claimed a proof
of the $n!$ theorem.  This work has been funded by the  Russian Academic
Excellence Project '5-100'.

\section{Preliminaries on Hilbert schemes}\label{S_Hilbert}
\subsection{Hilbert scheme $H_n$}\label{SS_Hilbert_general}
First of all, let us recall a construction of $H_n$ via the Hamiltonian reduction. Consider the space
$R:=\operatorname{End}(\C^n)\oplus \C^n$. This space has a natural action of $G:=\operatorname{GL}_n$.

Consider the cotangent bundle $T^*R=R\oplus R^*$. The group $G$ acts on $T^*R$ and the action is Hamiltonian.
The moment map can be described as follows. First, we can identify $\End(\C^n)^*$ with $\End(\C^n)$
via the trace form. Then we can view an element of $T^*R$ as a quadruple $(A,B,i,j)$ for $A,B\in \End(\C^n),
i\in \C^n, j\in \C^{n*}$. The moment map $\mu:T^*R\rightarrow \g$ is given by $\mu(A,B,i,j)=AB+ij$.
Now we are going  to describe the pull-back map $\mu^*:\g\rightarrow \C[T^*R]$.
Consider the velocity vector field map $\g\rightarrow \operatorname{Vect}(R), \xi\mapsto \xi_R$.
We can view $\operatorname{Vect}(R)$ as a subspace of $\C[T^*R]=\C[R\oplus R^*]$ that consists of all functions that have
degree $1$ in $R^*$. It is easy to show that  $\mu^*(\xi)= \xi_R$.

Consider the character $\theta:=\det^{-1}$. By definition, the $\theta$-stable locus in $T^*R$
consists of the quadruples $(A,B,i,j)$ such that $\C\langle A,B\rangle i=\C^n$. This locus
will be denoted by $(T^*R)^{\theta-s}$. Note that the action
of $G$ on $(T^*R)^{\theta-s}$ is free. For $(A,B,i,j)\in \mu^{-1}(0)^{\theta-s}:=\mu^{-1}(0)
\cap (T^*R)^{\theta-s}$ we necessarily have $j=0$ and hence $[A,B]=0$.

By definition, $H_n$ is the GIT Hamiltonian reduction of $T^*R$ by the action of $G$
with character $\theta$. Equivalently, $H_n$ is the GIT quotient $\mu^{-1}(0)^{\theta-s}/G$. Below we will
write $\Str_n$ for the structure sheaf $\Str_{H_n}$ of $H_n$.
We have an ample line bundle $\mathcal{O}_n(1)$
on $H_n$ given by $\varphi_*(\Str_{\mu^{-1}(0)^{\theta-s}}\otimes \det)^G$, where we write
$\varphi$ for the quotient morphism $\mu^{-1}(0)^{\theta-s}\rightarrow H_n$.

We have a natural morphism of quotients $\rho_n: H_n\rightarrow \mu^{-1}(0)/\!/G$. The target variety is identified
with $V_n/S_n$ and the morphism is the Hilbert-Chow map that sends a codimension $n$ ideal
to its support with multiplicities. Also recall that $H_n$ is a resolution of singularities
of $V_n/S_n$. The morphism $\rho_n$ is an isomorphism over the smooth locus $(V_n/S_n)^{reg}$.
This locus coincides with $V_{n}^{0}/S_n$, where we write $V_n^0$ for the locus of $n$ different points in $\C^2$.

\begin{Ex}\label{Ex_Hilb2}
It is easy to see that $H_2=\operatorname{Bl}_\Delta(V_2/S_2)$, the blow-up
of the diagonal.
\end{Ex}

Also note that $H_n$ comes with several additional structures. For examples, as a Hamiltonian
reduction of a smooth symplectic variety by a free $G$-action, $H_n$ carries a natural symplectic
form, $\omega$. Since $H_n$ is a symplectic resolution of a normal variety, by the Grauert-Riemenschneider
theorem, we have the following claim.

\begin{Lem}\label{Lem:Hilbert_cohom_vanishing}
We have $H^i(H_n, \Str_{n})=0$ for $i>0$ and $\Gamma(\Str_{n})=\C[V_n]^{S_n}$.
\end{Lem}

The variety $H_n$ also comes with an action of the two-dimensional torus $T$. It is induced from
the $T$-action on $T^*R$ given by $(t_1,t_2)(r,\alpha)=(t_1^{-1}r, t_2^{-1}\alpha)$ for
$r\in R,\alpha\in R^*$. For the induced action on $H_n$ we have $(t_1,t_2).\omega=t_1t_2\omega$.
In particular, the action of $T_h:=\{(t,t^{-1})\}\subset T$ preserves $\omega$ and, moreover,
is Hamiltonian. The action of $T_c:=\{(t,t)\}\subset T$ is contracting and rescales
the symplectic form $\omega$ by $t^2$. Below this action will be  called {\it contracting}.

The $T$-fixed points correspond to the monomial ideals in $\C[x,y]$ and hence
are labelled by the partitions of $n$: for a partition $\lambda$ of $n$, we write
$x_{\lambda}$ for the corresponding fixed point.

Now let us recall how to compute the group $\operatorname{Pic}(H_n)$ and the space
$H^2(H_n,\C)$. Pick the subtorus of the form $\{(t,t^N)\}\subset T$ for $N\gg 0$, it is contracting
and has finitely many fixed points.   The open cell for the Bialynicki-Birula decomposition corresponds to the  fixed
point $x_{(n)}$. And there is one codimension $1$ cell corresponding
to  $x_{(n-1,1)}$. It follows that $\operatorname{Pic}(H_n)\cong \Z$, in fact, this group is
generated by the line bundle $\Str_{n}(1)$ constructed above. And $H^2(H_n,\C)\cong \C$
is generated by $c_1(\Str_{n}(1))$. Note that the closure of the codimension 1 cell coincides
with $H_n\setminus \rho_n^{-1}(V_n^0/S_n)$.

Let us describe the structure of $H_n$ over a neighborhood of a generic point in $(V_n/S_n)^{sing}$.
The following lemma is classical.

\begin{Lem}\label{Lem:Hilb_local}
Let $b\in V_n/S_n$ be an $n$-tuple with precisely one pair of repeated points
and $(V_n/S_n)^{\wedge_b}$ denote its formal neighborhood in $V_n/S_n$.
Then $\rho_n^{-1}\left((V_n/S_n)^{\wedge_b}\right)$ is isomorphic to the preimage of $(V_2/S_2\times V_{n-2})^{\wedge_0}$  in $\operatorname{Bl}_{\Delta}(V_2/S_2)\times
V_{n-2}$.
\end{Lem}

To finish the section, let us discuss possible rings of definition. The scheme $H_n$ is defined
over $\Z$. For general reasons, the corresponding $\Ring$-form  $H_n^{\Ring}$
is smooth and symplectic over a finite localization
$\Ring$ of $\Z$. Further it is a resolution of singularities of $V_n^\Ring/S_n$.
We will also change the base from $\Ring$ to $\F:=\overline{\F}_p$ for $p\gg 0$
and all results in the previous paragraphs of this section continue to hold.


\subsection{Nested Hilbert scheme $H_{n,n-1}$}
Recall the nested Hilbert scheme $H_{n,n-1}$ parameterizing pairs of ideals $J\subset J'\subset \C[x,y]$
such that $\dim \C[x,y]/J=n, \dim \C[x,y]/J'=n-1$.  We have the following
fundamental result on $H_{n,n-1}$, \cite[Section 0.2]{Cheah} (the theorem was actually proved earlier
by A. Tikhomirov in an unpublished preprint).

\begin{Prop}\label{Prop:two_step_smooth}
The scheme $H_{n,n-1}$ is smooth and irreducible of dimension $2n$.
\end{Prop}

Here and below to simplify the notation we will write $H_{n-1,\diamondsuit}$
for $H_{n-1}\times \C^2$.
As was mentioned in the introduction, the variety $H_{n,n-1}$ comes with two morphisms:
$\alpha_n: H_{n,n-1}$ that sends $(J\subset J')$ to $J$ and $\beta_{n}: H_{n,n-1}\rightarrow
H_{n-1,\diamondsuit}$, which sends $(J\subset J')$ to the pair of $J'$ and the support of $J'/J$.

Consider the  morphism $\rho_{n,n-1}:H_{n,n-1}\rightarrow V_n/S_{n-1}$
given by $\rho_{n-1}\circ \beta_n$. Note that it is an isomorphism over $V_n^0/S_{n-1}$.

\begin{Lem}\label{Lem:nested_global}
The morphism $\rho_{n,n-1}$ gives rise to an isomorphism
$\C[V_n]^{S_{n-1}}\cong \C[H_{n,n-1}]$. Moreover, the composition $\rho_n\circ\alpha_n:H_{n,n-1}
\rightarrow V_n/S_n$
factors as $H_{n,n-1}\xrightarrow{\rho_{n,n-1}} V_n/S_{n-1}\rightarrow V_n/S_n$.
\end{Lem}
\begin{proof}
The first claim holds because $V_n/S_{n-1}$ is normal and the morphism
$\rho_{n,n-1}$ is birational and projective.
The second claim is straightforward.
\end{proof}

\begin{Ex}\label{Ex:nested2}
Let us consider the case of $n=2$. Here $H_{1,\diamondsuit}=V_2$ and $H_{2}=\operatorname{Bl}_\Delta(V_2/S_2)$.
In fact, $H_{n,n-1}=\operatorname{Bl}_\Delta(V_2)$. The morphism $\beta_2$ is the natural
projection $\operatorname{Bl}_\Delta(V_2)\rightarrow V_2$. The morphism $\alpha_2$ is
the quotient morphism for the action of $S_2$ on $\operatorname{Bl}_\Delta(V_2)$
that is induced from $S_2$ acting on $V_2$.
\end{Ex}

The morphism $\alpha_n\times \beta_n: H_{n,n-1}\hookrightarrow
H_n\times H_{n-1,\diamondsuit}$ is an embedding.

\begin{Lem}\label{Lem:lagrangian}
The subvariety $H_{n,n-1}\subset H_{n}\times H_{n-1,\diamondsuit}$ is lagrangian, where we consider $H_n$
with the opposite symplectic form.
\end{Lem}
\begin{proof}
Note that $V^0_n/S_{n-1}$ is lagrangian in $V^0_n/S_n\times V^0_n/S_{n-1}$.
Recall that $\rho_{n,n-1}$ is an isomorphism over $V^0_n/S_{n-1}$. Since $H_{n,n-1}$
is irreducible, we see that it is lagrangian in $H_n\times H_{n-1,\diamondsuit}$.
\end{proof}

Thanks to Example \ref{Ex:nested2}, we have  the following analog of Lemma \ref{Lem:Hilb_local}.

\begin{Lem}\label{Lem:nested_Hilb_local}
Let $b\in V_n/S_n$ be as in Lemma \ref{Lem:Hilb_local}. Then
$\alpha_n^{-1}(\rho_n^{-1}\left( (V_n/S_n)^{\wedge_b}\right))$ is the disjoint union
of the following schemes:
\begin{itemize}
\item[(i)] $n-2$ copies of the preimage of $(V_2/S_2\times V_{n-2})^{\wedge_0}$  in $\operatorname{Bl}_{\Delta}(V_2/S_2)\times
V_{n-2}$, each mapping to $\rho^{-1}\left( (V_n/S_n)^{\wedge_p}\right)$ isomorphically,
\item[(ii)] and one copy of the preimage of $V_n^{\wedge_0}$ in $\operatorname{Bl}_{\Delta}(V_2)\times
V_{n-2}$, which maps to $\rho^{-1}\left( (V_n/S_n)^{\wedge_p}\right)$ via the quotient map
for the $S_2$-action.
\end{itemize}
\end{Lem}

Now we compute $\operatorname{Pic}(H_{n,n-1})$ and $H^2(H_{n,n-1},\C)$.

\begin{Lem}\label{Lem:Pic_nested}
Let $n>2$. Then the  group $\operatorname{Pic}(H_{n,n-1})$ is a free abelian group with basis $\alpha_n^*\Str_{n}(1)$
and $\beta_n^* \Str_{n-1,\diamondsuit}(1)$. The space $H^2(H_{n,n-1},\C)$ is two-dimensional
with basis $$c_1(\beta_n^* \Str_{n-1,\diamondsuit}(1)), c_1(\alpha_n^*\Str_{n}(1)).$$
\end{Lem}
\begin{proof}
Consider the divisors $D_1,D_2$ in $H_{n,n-1}$ defined as follows: $D_1$ is the locus
of $(J'\subset J)$ such that the support of $\C[x,y]/J'$ has repeated points, while $D_2$
is the locus, where the support of $J'/J$ is contained in that of $\C[x,y]/J'$.  Note that $D_1\cap D_2$
has codimension  $2$, while $D_1\cup D_2=H_{n,n-1}\setminus \rho_{n,n-1}^{-1}(V^0_n/S_{n-1})$,
the locus of $(J'\subset J)$ such that the support of $J$ has repeated points.
Both $D_1,D_2$ are irreducible.

Note that $\beta_n$ is smooth at a generic point of $D_1$,
while from Lemma \ref{Lem:nested_Hilb_local} it follows that $\alpha_n$ is smooth at the generic points of  $D_1, D_2$.
Therefore $\beta_n^* \Str_{H_{n-1,\diamondsuit}}(1)=\mathcal{O}(D_1)$
and $\alpha_n^* \Str_{H_n}(1)=\mathcal{O}(D_1+D_2)$.
So what remains to prove is that $\operatorname{Cl}(H_{n,n-1})$ is
generated by $D_1,D_2$. Then it is freely generated.

Consider the natural action of $T$ on $H_{n,n-1}$. It contains a contracting torus and has
finitely many fixed points (because the actions of $T$ on $H_{n-1}$ and $H_n$
have these properties and we have $H_{n,n-1}\hookrightarrow H_n\times H_{n-1}$).
The points in $H_{n,n-1}^T$ are labelled by Young diagrams with $n$ boxes and fixed
corner box (so that we get a diagram with $n-1$ boxes by removing the fixed box).
The $T$-action on the tangent spaces was computed in \cite[Section 2]{Cheah}. In particular,
take the one-parameter subgroup of $T$ of the form $t\mapsto (t,t^N)$ for
$N\gg 0$. From \cite[Proposition 2.6.4]{Cheah} it follows that
the open cell for the Bialynicki-Birula decomposition corresponds to the unique fixed
point with diagram $(n)$ and there are two codimension $1$ cells corresponding
to the two possible corner boxes of $(n-1,1)$. Similarly to the case of usual
Hilbert schemes, the closures of these cells are
$D_1$ and $D_2$. So these divisors freely generate $\operatorname{Cl}(H_{n,n-1})$
and the proof for the Picard group is finished. The proof for the second cohomology
space follows.
\end{proof}

%

Following Haiman, we will write
$$\Str_{n,n-1}(k,l):=\beta_n^*(\Str_{n-1,\diamondsuit}(k))\otimes \alpha_n^* \Str_{n}(\ell).$$
The following result is due to Haiman, \cite[Proposition 3.6.4]{Haiman_n!}.

\begin{Lem}\label{Lem:nested_canonical}
For $n>2$, the canonical bundle $K_{H_{n,n-1}}$ is $\Str_{n,n-1}(1,-1)$.
\end{Lem}

To finish the section, let us note that Proposition \ref{Prop:two_step_smooth},
Example \ref{Ex:nested2}, Lemmas \ref{Lem:nested_global}, \ref{Lem:lagrangian}, \ref{Lem:nested_Hilb_local},
\ref{Lem:nested_canonical}, as well as the computation of the Picard group
in Lemma \ref{Lem:Pic_nested}  still hold if we replace $\C$ with a sufficiently large
finite localization $\Ring$ of $\Z$. Hence they also hold over $\F:=\overline{\F}_p$
for $p\gg 0$.

\subsection{Universal family $U_n$}
The variety $H_n$ parameterizes the codimension $n$ ideals in $\C[x,y]$.
In particular, it comes with the sheaf of algebras whose fiber at $I\in H_n$ is $\C[x,y]/I$.
We denote this sheaf of algebras by $\mathcal{T}_n$, this is a rank $n$ vector bundle.

Let $U_n$ denote the relative spectrum $\operatorname{Spec}_{\Str_n}(\mathcal{T}_n)$.
By the construction, $U_n$ is a closed subscheme of $H_n\times \C^2$.
It admits a finite degree $n$ morphism to $H_n$, to be denoted by $\tau_n$ so that
$\mathcal{T}_n:=\tau_{n*}\Str_{U_n}$.
We have
$\mathcal{T}_n:=\pi_*(\Str_{\mu^{-1}(0)^{\theta-s}}\otimes \C^n)^G$, so, in particular,
$\Str_{n}(1)=\Lambda^n \mathcal{T}_n$.

\begin{Ex}\label{Ex:universal2}
For $n=2$, we have $U_2=\operatorname{Bl}_\Delta(V_2)$ and $\tau_2$ is the quotient morphism
for the $S_2$-action. In particular, $U_2=H_{2,1}$.
\end{Ex}

We see that, in general, $U_n$ is a
Cohen-Macaulay scheme. Note also that $\tau_n^{-1}(V_n^0/S_n)=V_n^0/S_{n-1}$.
In particular, $U_n$ is generically reduced and since it is Cohen-Macaulay, we see that $U_n$ is reduced.

Let us write $V_n^1$ for the locus of points in $V_n$ with no more than one pair of repeated
points.

\begin{Lem}\label{Lem:alpha_decomp}
The following claims are true:
\begin{enumerate}
\item We have $\alpha_n=\tau_n\circ \bar{\alpha}_n$ for a projective birational morphism
 $\bar{\alpha}_n: H_{n,n-1}\rightarrow U_n$.
\item $\bar{\alpha}_n$ is an isomorphism over $V_n^1/S_n$.
\item $U_n$ is normal and $\alpha_n=\tau_n\circ \bar{\alpha}_n$
is the Stein factorization of $\alpha_n$.
\end{enumerate}
\end{Lem}
\begin{proof}
To prove (1) we note that $\alpha_n$ naturally factors as $H_{n,n-1}\rightarrow H_n\times
\C^2\rightarrow H_n$, where the first morphism is projective.
That morphism factors as $H_{n,n-1}\rightarrow U_n\hookrightarrow H_n\times \C^2$.
We take this arrow for $\bar{\alpha}_n$.
Over $V^0_n/S_n$, the morphism $\bar{\alpha}_n$ is the isomorphism $V^0_n/S_{n-1}\xrightarrow{\sim}
V^0_n/S_{n-1}$. This proves (1).

Let us prove (2). By Example \ref{Ex:universal2}, $\bar{\alpha}_2$ is an isomorphism.
In the notation of Lemma \ref{Lem:Hilb_local},
$\tau_n^{-1}(\rho_n^{-1}(V_n/S_n)^{\wedge_b})$ admits the same description as
$\alpha_n^{-1}(\rho_n^{-1}(V_n/S_n)^{\wedge_b})$, whose description was given Lemma \ref{Lem:nested_Hilb_local}. This proves (2).

To prove (3), we notice that, by (2), $U_n$ is smooth outside of codimension $2$.
Since $U_n$ is Cohen-Macaulay, it is normal. Hence $\alpha_n=\tau_n\circ \bar{\alpha}_n$
is the Stein factorization.
\end{proof}

\begin{Lem}\label{Lem:universal_global}
We have $\C[U_n]=\C[V_{n}]^{S_{n-1}}$, which gives rise to the morphism $\bar{\rho}_n:U_n
\rightarrow V_n/S_{n-1}$.  The morphism $\rho_n\circ\tau_n:U_n
\rightarrow V_n/S_n$ factors as $U_n\xrightarrow{\bar{\rho}_n} V_n/S_{n-1}\rightarrow V_n/S_n$.
\end{Lem}
\begin{proof}
By (3) of Lemma \ref{Lem:alpha_decomp}, $\C[U_n]=\C[H_{n,n-1}]$. Now the claims of this lemma
follow from Lemma \ref{Lem:nested_global}.
\end{proof}

Now we are going to describe the restriction of $\mathcal{T}_n$ to $\rho_n^{-1}((V_n/S_n)^{\wedge_b})$, where
 we use the notation of Lemma \ref{Lem:Hilb_local}.

\begin{Lem}\label{Lem:taut_restriction}
We have
$$\mathcal{T}_n|_{\rho_n^{-1}((V_n/S_n)^{\wedge_b})}\cong \Str_{\rho_n^{-1}((V_n/S_n)^{\wedge_b})}^{\oplus n-1}
\oplus \Str_{\rho_n^{-1}((V_n/S_n)^{\wedge_b})}(1).$$
\end{Lem}
\begin{proof}
Lemma \ref{Lem:nested_Hilb_local} together with (2) of Lemma \ref{Lem:alpha_decomp} gives
a description of $$\tau_n: \tau_n^{-1}(\rho_n^{-1}((V_n/S_n)^{\wedge_b}))
\rightarrow \rho^{-1}((V_n/S_n)^{\wedge_b}.$$ In particular, each of the $n-1$
components of $\tau_n^{-1}(\rho_n^{-1}((V_n/S_n)^{\wedge_b}))$  contributes
one copy of $\Str_{\rho_n^{-1}((V_n/S_n)^{\wedge_b})}$ to the restriction of
$\mathcal{T}_n$. There is one more summand:
the sign invariant part of the push-forward from the component (ii) in Lemma \ref{Lem:nested_Hilb_local}.
This is a line bundle that is forced to be $\Str_{\rho_n^{-1}((V_n/S_n)^{\wedge_b})}(1)$
because of the isomorphism $\Lambda^n \mathcal{T}_n\cong \Str_{n}(1)$.
\end{proof}

Similarly to the previous section, the results of this one hold over
$\Ring$ and hence over $\F$.

\section{Preliminaries on quantizations and Procesi bundles}
\subsection{Quantizations}\label{SS_quant_general}
Let us start with a general situation. Let $(X,\omega)$ be a smooth symplectic variety over $\C$.
In particular, the structure sheaf $\Str_X$ carries a Poisson bracket $\{\cdot,\cdot\}$ induced by $\omega$.
A  {\it formal quantization} $\mathcal{D}_\hbar$ of $X$ is a sheaf
of $\C[[\hbar]]$-algebras on $X$ (in the Zariski topology) together with an isomorphism $\kappa:\mathcal{D}_\hbar/(\hbar)\xrightarrow{\sim}
\Str_X$ such that
\begin{itemize}
\item[(a)]  $\mathcal{D}_\hbar$ is flat over $\C[[\hbar]]$ (i.e., there are no nonzero local sections
annihilated by $\hbar$) and complete and separated in the $\hbar$-adic topology (meaning that
$\mathcal{D}_\hbar\xrightarrow{\sim}\varprojlim_{n\rightarrow+\infty} \mathcal{D}_\hbar/(\hbar^n)$, where the inverse
limit is taken in the category of sheaves).
\item[(b)] $\kappa(\frac{1}{\hbar}[a,b])=\{\kappa(a),\kappa(b)\}$ for any local sections $a,b$ of $\mathcal{D}_\hbar$.
\end{itemize}

Assume now that $X$ carries an action of $\C^\times$ such that $t.\omega=t^d \omega$.
By a {\it grading} on $\D_\hbar$ we mean an action of $\C^\times$ on $\D_\hbar$ by sheaf of algebras
automorphisms such that $t.\hbar=t^d \hbar$ and $\varpi$ is $\C^\times$-equivariant. By a graded formal
quantization we mean a formal quantization together with a grading. From here we can define
a sheaf of filtered algebras in the conical topology (where  ``open''  means Zariski open
$\C^\times$-stable). Namely, for such a subset $U$, we set $\D(U):=D_\hbar(U)_{fin}/(\hbar-1)$,
where we write $\D_\hbar(U)_{fin}$ for the subalgebra of $\C^\times$-finite elements in
$\D_\hbar(U)$. It is easy to see that the algebras $\D(U)$ form a sheaf in the conical topology
to be denoted by $\D$. This sheaf comes with
\begin{itemize}
\item a complete and separated filtration (induced
by the $\hbar$-adic filtration on $\D_\hbar$)
\item and  an isomorphism $\gr\D\xrightarrow{\sim} \Str_{X}$
of sheaves of graded Poisson algebras in the conical topology (induced by $\kappa$).
\end{itemize}

By  a {\it microlocal filtered quantization} of $\Str_X$ we mean a sheaf $\D$ of algebras in the conical
topology on $X$ with the two additional structures above. Then we can recover a graded formal quantization
from $\D$ by taking the $\hbar$-adically completed Rees sheaf.

We will need a classification result for the formal graded (equivalently, microlocal filtered)
quantizations under certain cohomology
vanishing assumptions that was obtained in \cite[Section 2.3]{quant_iso} as a ramification of a classification
result for all formal quantizations from \cite{BK_quant}. Namely, in \cite[Section 4]{BK_quant}
Bezrukavnikov and Kaledin introduced the noncommutative period map $\mathsf{Per}$ from the set of isomorphism classes of formal quantizations to $H^2(X,\C)[[\hbar]]$. Further, it was shown in \cite[Section 2.3]{quant_iso} that for graded
formal quantizations, $\mathsf{Per}$ takes values in $H^2(X,\C)$, the subspace of elements
of $H^2(X,\C)[[\hbar]]$ independent of $\hbar$.

\begin{Lem}\label{Lem:gr_form_quant_classif}
Assume that $H^i(X,\Str_X)=0$ for $i=1,2$. Then  $\mathsf{Per}$ is  an isomorphism between the set of isomorphism classes
of  graded formal quantizations and $H^2(X,\C)$.
\end{Lem}

Note that if $\D$ is a filtered microlocal quantization of $X$, then $\D^{opp}$ is naturally a filtered
microlocal quantization of $X^{opp}$, the same variety as $X$ but with  symplectic
form $-\omega$. By \cite[Section 2.3]{quant_iso}, we have
\begin{equation}\label{eq:period_opposite}
\mathsf{Per}(\D^{opp})=-\mathsf{Per}(\D).
\end{equation}

A microlocal quantization $\D(=\D_n)$ of $H_n$ that we need is constructed in the following
way. We first consider the microlocalization $D_R$ of the algebra $D(R)$
of differential operators to $T^*R$. The sheaf $D_R$ is
a microlocal filtered quantization of $T^*R$.  Consider the quantum comoment map $\Phi:\g\rightarrow D(R)$ given
by $\xi\mapsto \xi_R$. Then  we form the quotient $D_R/D_R \Phi(\g)$, it is scheme-theoretically
supported on $\mu^{-1}(0)$. Recall that $\varphi$ denotes the quotient map $\mu^{-1}(0)^{\theta-s}
\rightarrow H_n$. We set
$$\D:=(\varphi_*\left([D_R/D_R\Phi(\g)]|_{(T^*R)^{\theta-s}}\right))^G.$$
This is a filtered microlocal quantization on $H_n$.

Let us now discuss the global sections of $\D$. Let us write $\A_n$ for $\Weyl(V_n)^{S_n}$.

\begin{Lem}\label{Lem:glob_sections}
We have $R\Gamma(\D)=\A_n$, i.e.,  $H^i(H_n,\D)=0$ for $i>0$ and $\Gamma(\D)=\A_n$.
\end{Lem}
\begin{proof}
The cohomology vanishing part follows from $H^i(H_n, \Str_{n})=0$, see Lemma
\ref{Lem:Hilbert_cohom_vanishing}. By \cite[Lemma 4.2.4]{quant_iso},
$\Gamma(\D)$ coincides with the global Hamiltonian reduction $[D(R)/D(R)\Phi(\g)]^G$. The latter
is isomorphic to $\A_n$ by results of \cite[Section 6]{GG}.
\end{proof}

Lemma \ref{Lem:gr_form_quant_classif} applies to $X=H_n$. Recall, Section \ref{SS_Hilbert_general}, that
$H^2(H_n,\C)$ is identified with $\C$.
Consider the symmetrized quantum comoment map $\Phi^{sym}:\g\rightarrow D(R)$ given by
$\Phi^{sym}(\xi)=\frac{1}{2}(\xi_R+\xi_{R^*})$. Note that $\xi_{R^*}=\xi_R+\operatorname{tr}(\xi)$
so $\Phi(x)=\Phi^{sym}(\xi)-\frac{1}{2}\operatorname{tr}(\xi)$.
By \cite[Section 5]{quant_iso}, we have
\begin{equation}\label{eq:period_quantization}
\mathsf{Per}(\D)=-\frac{1}{2}.
\end{equation}

Now we proceed to discussing quantizations in positive characteristic.

We can do the same  construction  of quantum Hamiltonian reduction
in characteristic $p$ getting a microlocal filtered sheaf of
algebras $\D^\F$ on $H^\F_n$. In fact, this can be done over a sufficiently large finite
localization $\Ring$ of $\Z$: we get a microlocal sheaf of algebras $\D^\Ring$
on $H_n^{\Ring}$ with $\D=\C\widehat{\otimes}_{\Ring}\D^\Ring$ and $\D^\F=\F\widehat{\otimes}_{\Ring}\D^\Ring$
(we take the completed tensor product with respect to the topology defined by the filtration).
As was checked in \cite[Lemma 4.4]{Procesi_review}, $R\Gamma(\D^\Ring)=\A_n^\Ring$.
Hence $R\Gamma(\D^\F)=\A_n^\F$.

We can also construct a version of a quantization that is a genuine coherent sheaf on
the Frobenius twist $H_{n}^{\F(1)}$ following \cite[Section 4]{BFG}. Namely, $D(R^{\F})$
is an Azumaya algebra on $T^*R^{\F(1)}$. Consider the restriction
$D(R^\F)|_{(\mu^{(1)})^{-1}(0)^{\theta-s}}$. The Lie algebra $\g^\F$ acts
on this Azumaya algebra by $\Str_{(\mu^{(1)})^{-1}(0)^{\theta-s}}$-linear
derivations. It was shown in \cite[Section 4]{BFG} that the sheaf of invariants
$[D(R^\F)|_{(\mu^{(1)})^{-1}(0)^{\theta-s}}]^{\g}$ is also an Azumaya
algebra, it is $G^{\F(1)}$-equivariant.  Now we can define the sheaf
$\underline{\D}^\F$ by
$$\underline{\D}^\F:=(\varphi_*^{(1)}\left([D(R^\F)|_{(\mu^{(1)})^{-1}(0)^{\theta-s}}]^{\g}\right))^{G^{\F(1)}}.$$
This is a sheaf of Azumaya algebras on $H_{n}^{\F(1)}$ of rank $p^{2n}$. Its restriction to the conical
topology comes with a filtration such that  $\gr \underline{\D}^\F=\operatorname{Fr}_*\Str^\F_{n}$,
where we write $\operatorname{Fr}$ for the Frobenius morphism $H_n^\F\rightarrow H_n^{\F(1)}$.
So it is a Frobenius-constant quantization in the terminology of \cite[Section 3.3]{BK}.
We have $R\Gamma(\underline{\D}^\F)=\A_n^\F$ by \cite[Theorem 4.1.4]{BFG}.

Let us explain how to pass from $\underline{\D}^\F$ to $\D^\F$. By the construction,
we have a homomorphism $\underline{\D}^\F\rightarrow \Fr_* \D^\F$ of sheaves of filtered
algebras on $H_n^{\F(1)}$. The associated graded homomorphism is the identity.
In particular, it is strictly compatible with filtrations. So we see that
$\Fr_*\D^\F$ is the completion of $\underline{\D}^\F$.

\subsection{Bezrukavnikov-Kaledin construction of Procesi bundles}\label{SS_Procesi_constr}
In \cite[Section 6]{BK}, Bezrukavnikov and Kaledin used the Azumaya algebra $\underline{\D}^\F$
on $H_n^{\F(1)}$ to construct a Procesi bundle. The Azumaya algebra $\underline{\D}^\F$ does not
split. However, its restriction $\underline{\D}^{\F\wedge}$ to $H_{n}^{\F(1),\wedge}:=\rho_n^{-1}((V_n^{\F(1)}/S_n)^{\wedge})$,
where $(V_n^{\F(1)}/S_n)^{\wedge}$ denotes  the formal neighborhood of $0$
in $V_n^{\F(1)}/S_n$,  does. This was proved in \cite[Section 6.3]{BK} (see also \cite[Section 6.2]{BL}
for a different proof). Let $\hat{\E}^\F$ denote a splitting bundle. Note that the endomorphism
algebra of $\hat{\E}^\F$ is identified with $\A_n^{\F\wedge}$ (the completion
at $0\in V_{n}^{\F(1)}/S_n$). And there are no higher self-extensions of $\hat{\E}^\F$,
thanks to the formal function theorem and $R^i\Gamma(\underline{\D}^\F)=0$ for $i>0$.

The Azumaya algebra $\Weyl(V_n^\F)^{\wedge}$ admits a standard $S_n$-equivariant splitting
module leading to a Morita equivalence $\A_n^{\F\wedge}\cong \widetilde{A}_n^{\F\wedge}$.
Here and below we write $\widetilde{A}_n^\F$ for the algebra $\F[V_n]\# S_n$.
Also $\Weyl(V_n^\F)\# S_n$ and $\A_n^\F$
are Morita equivalent via the bimodule $\Weyl(V_n^\F)$. It follows that the algebras
$\A_n^{\F\wedge}$ and $\widetilde{A}_n^{\F\wedge}$ are Morita equivalent. Let us write
$\hat{\B}_n^\F$ for the Morita equivalence $\widetilde{A}_n^{\F\wedge}$-$\A_n^{\F\wedge}$-bimodule,
and $\hat{\B}_n^{o,\F}$ for the inverse bimodule. In particular, we have
$$\widetilde{A}_n^{\F\wedge}=\hat{\B}_n^{\F}\otimes_{\A_n^{\F\wedge}}\hat{\B}_n^{o,\F}.$$

We set (basically following \cite[Section 4.3]{SRA_der})
\begin{equation}\label{eq:Procesi_compl_definition}
\hat{\Pro}^\F_n:=\hat{\B}_n^\F\otimes_{\A_n^{\F\wedge}}\hat{\E}^\F.
\end{equation}
Since $\hat{\B}_n^{\F}$ is projective as a right $\A_n^{\F\wedge}$-module, we see that
$\hat{\Pro}^\F_n$ has no higher self-extensions and
$$\End(\hat{\Pro}_n^\F)=\hat{\B}_n^\F\otimes_{\A_n^{\F\wedge}}\End(\hat{\E}_\F)
\otimes_{\A_n^{\F\wedge}}\hat{\B}_n^{o,\F}=\widetilde{A}_n^{\F\wedge}.$$

Also, since $\hat{\Pro}^\F_n$ has no higher self-extensions, it admits a
$T_c^{\F(1)}$-equivariant structure, see \cite{Vologodsky}. Since the $T_c^{\F(1)}$-action on $H_n^{\F(1)}$
is contracting, the completion functor is an equivalence $$\Coh^{T_c^{(1)}}(H_n^{\F(1)})
\xrightarrow{\sim} \Coh^{T_c^{(1)}}(H_n^{\F(1)\wedge}).$$
Let $\Pro^\F_n$ be the $T_c^{\F(1)}$-equivariant sheaf on $H_n^{\F(1)}$ corresponding
to $H_n^{\F(1)\wedge}$. Note that as a sheaf $\Pro_n^\F$ is independent of the
choice of an equivariant structure on $\hat{\Pro}_n^\F$. Moreover, $\Pro_n^\F$
has no higher self-extensions, and we have $\End(\Pro_n^\F)\cong \widetilde{A}_n^{\F(1)}$.
We can twist $\Pro_n^\F$ with a line bundle and achieve that $\left( \Pro_n^\F\right)^{S_n}=\Str_{n}^{\F(1)}$.

Let us also point out that, since $H_n^\F$ is defined over $\F_p$, we have $H_n^{\F(1)}\cong H_n^\F$.
So we can view $\Pro_n^\F$ as a vector bundle over $H_n^\F$.

Then one can lift $\Pro_n^\F$ to characteristic $0$ as described in \cite[Sections 2.3, 6.4]{BK}
getting a Procesi bundle $\Pro_n$ on $H_n$.

\subsection{Procesi bundles, axiomatically}
By a Procesi bundle on $H_n$ we mean a $T_c$-equivariant vector bundle $\Pro$ together with
a graded $\C[V_n]^{S_n}$-algebra isomorphism $\End(\Pro)\xrightarrow{\sim} \widetilde{A}_n^{opp}$  satisfying
\begin{itemize}
\item[(i)] $\Ext^i(\Pro,\Pro)=0$ for all $i>0$,
\end{itemize}
By an isomorphism  of two Procesi bundles $\Pro,\Pro'$ we mean a $\C^\times$-equivariant
isomorphism $\Pro\xrightarrow{\sim} \Pro'$ such that the corresponding isomorphism
$\widetilde{A}_n\xrightarrow{\sim} \widetilde{A}_n$ is inner (this algebra  has a graded outer automorphism given by $\alpha\mapsto \alpha, \sigma\mapsto (-1)^{\sigma}\sigma$
for $\alpha\in V^*_n, \sigma\in S_n$). Thanks to the isomorphism $\End(\Pro)\xrightarrow{\sim} \widetilde{A}_n^{opp}$,
every fiber of $\Pro$ is the regular representation of $S_n$. It follows that $\Pro^{S_n}$ (as well
as the sign component $\Pro^{\operatorname{sgn}}$) is a line bundle.

By a {\it normalized} Procesi bundle we mean a Procesi bundle satisfying the additional condition:
\begin{itemize}
\item[(ii)] $\Pro^{S_n}\cong \Str_{n}$ (an isomorphism of $T_c$-equivariant vector bundles).
\end{itemize}
We can twist every Procesi bundle with a line bundle (and take the induced isomorphism
$\End(\Pro)\xrightarrow{\sim} \widetilde{A}_n$) to achieve (ii).

In particular,  the Bezrukavnikov-Kaledin bundle $\Pro_n$ recalled in
Section \ref{SS_Procesi_constr}, is a normalized Procesi bundle.

The following proposition is a special case of \cite[Theorems 1.1,1.2]{L_Procesi}.
Note that we have an anti-involution of $\tilde{A}_n$ defined by $\alpha\mapsto \alpha,
\sigma\mapsto \sigma^{-1}$.
So the dual of a (normalized) Procesi bundle is also a (normalized) Procesi bundle.

\begin{Prop}\label{Prop:Procesi_classif}
The following statements hold.
\begin{enumerate}
\item There are exactly two distinct normalized Procesi bundle on $H_n$.
\item They are dual to each other.
\item For one of these Procesi bundles we have $\Pro^{S_{n-1}}\cong \mathcal{T}_n$.
\end{enumerate}
\end{Prop}

The following lemma  can be proved similarly to (1),(3) of Proposition \ref{Prop:Procesi_classif},
but we will deduce it from this proposition and another result of \cite{L_Procesi}.

\begin{Lem}\label{Lem:Procesi_sgn}
If $\Pro^{S_{n-1}}\cong \mathcal{T}_n$, then $\Pro^{\sgn}\cong \Str_n(1)$.
\end{Lem}
\begin{proof}
Let $b\in V_n/S_n$ be as in Lemma \ref{Lem:Hilb_local}. By \cite[Proposition 4.1]{L_Procesi},
the restriction of $\Pro$ to $\rho_n^{-1}((V_n/S_n)^{\wedge_b})$ is isomorphic to the direct
sum of $n!/2$ copies of a Procesi bundle on $\rho_2^{-1}((V_2/S_2)^{\wedge_0})$. The latter is
isomorphic to $\Str^{\wedge}_2\oplus \Str^\wedge_2(1)$ or $\Str^\wedge_2\oplus \Str^\wedge_2(-1)$.
Since $\Pro^{S_{n-1}}\cong \mathcal{T}_n$,
Lemma \ref{Lem:taut_restriction} implies that the restriction is $(\Str_2^\wedge\oplus \Str_2^\wedge(1))^{\oplus n!/2}$.
But the restriction of $\Str_n(k)$ is $\Str_2^\wedge(k)$ for all $k\in \Z$.
It follows that $\Pro^{\sgn}\cong \Str$ or $\Str(1)$. Let us show that the former is impossible.
Indeed, under the derived equivalence $R\Hom_{\Str_{H_n}}(\Pro,\bullet)$, the bundles
$\Pro^{S_n}, \Pro^{sgn}$ both map to $\C[V]$ but in the first case the action of $S_n$
is the natural one, and in the second case it is twisted with the sign. It is easy to
see that these two $\tilde{A}_n$-modules are not isomorphic, which implies $\Pro^{\sgn}\not\cong
\Pro^{S_n}$.
So $\Pro^{\sgn}\cong \Str_n(1)$.
\end{proof}

\begin{Cor}\label{Cor:Procesi_reln}
Let $\Pro$ be a Procesi bundle satisfying $\Pro^{S_{n-1}}\cong \mathcal{T}_n$. Then $\Pro\cong
\Pro^*(1)$, a $\C[V_n]$-linear isomorphism of vector bundles on $H_n$.
\end{Cor}
\begin{proof}
Twist the action of $\widetilde{A}_n$ on $\Pro^*$ with the outer automorphism mentioned above.
Then $\Pro^*(1)$ becomes normalized and also the sign components in $\Pro,\Pro^*(1)$
are isomorphic. A $\C[V_n]$-linear isomorphism $\Pro\cong
\Pro^*(1)$ now follows from Proposition \ref{Prop:Procesi_classif}
and Lemma \ref{Lem:Procesi_sgn}.
\end{proof}

\begin{Rem}\label{Rem:char_p_version}
We can define the notion of an abstract Procesi bundle over $\F(=\overline{\F}_p)$
for $p\gg 0$. The arguments of the proofs of \cite[Theorem 1.1,1.2]{L_Procesi}
carry over to this case without any significant modifications. Since $\Pro^{\sgn}$
is obtained by lifting its characteristic $p$ counterpart to characteristic $0$,
Lemma \ref{Lem:Procesi_sgn} continues to hold over $\F$. So does Corollary
\ref{Cor:Procesi_reln}.
\end{Rem}


\subsection{Quantization of lagrangian subvarieties}\label{SS_quant_lagrang}
In this section we will explain results from \cite{BGKP} on quantizations of
the structure sheaf of a smooth lagrangian subvariety $Y$ in a smooth
symplectic variety $X$ over $\C$.  Let $\mathcal{D}_\hbar$ be a formal
quantization of $\Str_X$. We want to know when there is a coherent sheaf
$\mathcal{M}_\hbar$ of $\D_\hbar$-modules that is flat over $\C[[\hbar]]$
and comes with an identification $\mathcal{M}_\hbar/\hbar \mathcal{M}_\hbar\cong \Str_Y$ of
$\Str_X$-modules. Then we call $\mathcal{M}_\hbar$ a formal quantization of $Y$.

Let $\iota$ denote the inclusion $Y\hookrightarrow X$ so that we have the pull-back map
$\iota^*: H^2(X,\C)\rightarrow H^2(Y,\C)$. Now assume that $H^i(Y,\Str_Y)=0$ for $i=1,2$.

Recall that we have the period $\mathsf{Per}(\D_\hbar)\in H^2(X,\C)[[\hbar]]$. The following is a special case of \cite[Theorem 1.1.4]{BGKP}.

\begin{Prop}\label{Prop:BGKP}
The $\Str_X$-module $\Str_Y$ admits a formal quantization $\M_\hbar$ if and only if
$$\iota^*(\mathsf{Per}(\D_\hbar))=-\frac{1}{2}c_1(K_Y).$$
\end{Prop}

Note that this  is true over any characteristic field $0$ field, which is already the generality
of \cite{BGKP}. 

We will be interested in microlocal filtered quantizations of $\Str_Y$. Let $\D$ be a microlocal
filtered quantization of $X$. Let $\mathcal{M}$ be a coherent sheaf of
$\D$-modules. By a {\it good filtration} on $\mathcal{M}$ we mean a $\D$-module filtration
$\mathcal{M}=\bigcup_{i\in \Z} \mathcal{M}_{\leqslant i}$ by sheaves of vector spaces subject to
the following two conditions:
\begin{itemize}
\item The filtration is complete and separated.
\item The associated graded $\Str_X$-module $\gr\mathcal{M}$ is coherent.
\end{itemize}

For example, if $\M_\hbar$ is a coherent $\D_\hbar$-module that comes with a $\C^\times$-action
compatible with that on $\D_\hbar$, then we can produce a coherent $\D$-module $\M:=\M_{\hbar,fin}/(\hbar-1)\M_{\hbar,fin}$. This module comes with a natural filtration, which is good.

A good filtration is far from being unique. However, any two good filtrations $(\M_{\leqslant i})_{i\in \Z}$ and $
(\M_{\preceq i})_{i\in \Z}$ of $\M$ are compatible in the following sense.

\begin{Lem}\label{Lem:good_filt_compat}
There are integers $d_1,d_2$ such that $\M_{\leqslant i-d_1}\subset \M_{\preceq i}\subset \M_{\leqslant i+d_2}$
for all $i$.
\end{Lem}
\begin{proof}
This can be checked locally, where we deal with good filtrations on modules over algebras.
In this case, the claim is classical.
\end{proof}

We have the following corollary of Proposition \ref{Prop:BGKP}.

\begin{Cor}\label{Cor:lag_quantization}
We assume that $X$ comes with a $T\times \C^\times$-action, where a torus $T$ preserves
$\omega$, and $\C^\times$ rescales $\omega$ as before. Let $\D$ be a microlocal
filtered quantization of $X$ with an action of a torus $T$ by filtered algebra automorphisms. Assume that $Y$
is $T$-stable and there is a formal quantization of $Y$. Then there is a filtered
quantization of $\Str_Y$ to a coherent $\D$-module $\M$ that carries an action
of $T$ that lifts the $T$-action on $\Str_Y$ and is compatible with the action of $T$ on $\D$.
\end{Cor}
\begin{proof}
Similarly to \cite[Section 2.3]{quant_iso}, one shows that the existence of a $T\times \C^\times$-action
on $\M_\hbar$ with required properties is equivalent to the claim that the isomorphism
class of the formal quantization $\M_\hbar$ is fixed by $T\times \C^\times$. Let $\mathsf{Quant}(Y,\D_\hbar)$
denote the set of isomorphism classes of formal quantizations of $\Str_Y$.

According to (2) of \cite[Theorem 1.1.4]{BGKP}, $\mathsf{Quant}(Y,\D_\hbar)$ is a torsor over the
group $\mathfrak{F}$ of $\Str_Y[[\hbar]]^\times$-torsors with a flat connection that are isomorphic
to the principal $\mathbb{G}_m$-bundle associated to $\Str_Y$
modulo $\hbar$. Note that $\mathfrak{F}$ is a vector space that is realized
as $\varprojlim_{k\rightarrow \infty} \mathfrak{F}_k$, where $\mathfrak{F}_k$ is the similarly defined
space of flat $(\Str_Y[[\hbar]]/(\hbar^k))^\times$-torsors. Each of the spaces $\mathfrak{F}_k$
comes with a rational action of $T\times \C^\times$. Moreover, for any element $\delta\in \mathsf{Quant}(Y,\D_\hbar)$,
the map $t\mapsto \phi_k(t): T\times \C^\times \rightarrow \mathfrak{F}_k$ defined by $\phi_k(t)\delta=t.\delta$
is algebraic. It follows that the cocycle $t\mapsto\phi(t)$ is actually a coboundary. Equivalently,
there is a $T\times \C^\times$-fixed point in $\mathsf{Quant}(Y,\D_\hbar)$.

\end{proof}

\section{Quantization of the nested Hilbert scheme}\label{S_quant_nested}
Recall that we write $H_{n-1,\diamondsuit}$ for $H_{n-1}\times \mathbb{A}^2$.
In Section \ref{SS_quant_general} we have constructed the filtered microlocal
quantizations of $H_{n},H_{n-1}$ to be denoted by $\D_n,\D_{n-1}$, both have period
$-\frac{1}{2}$. We also write $\D_{n-1,\diamondsuit}$ for $\D_{n-1}\widehat{\otimes}\Weyl(V_2)$,
this is a filtered microlocal quantization of $H_{n-1,\diamondsuit}$.

In this section we study a quantization of the $\Str_{n}$-$\Str_{n-1,\diamondsuit}$-bimodule
$\Str_{n,n-1}$ to a filtered coherent $\D_{n-1,\diamondsuit}$-$\D_n$-bimodule to be denoted by $\D_{n,n-1}$.

In Section \ref{SS_quant_0}, we  consider the quantizations over $\C$ and over $\Q$ using
results recalled in Section \ref{SS_quant_lagrang}. We will see that the filtered $\Gamma(\D_{n-1,\diamondsuit})$-$\Gamma(\D_n)$-bimodule of
global sections is equal to $\A_{n-1,\diamondsuit}$ and also show how to recover
$\D_{n,n-1}$ from its global sections.

In Section \ref{SS_quant_R},  we will produce a form of $\D_{n,n-1}$ over a finite localization
$\Ring$ of $\Z$ and show that it is still a microlocal filtered quantization of $\Str^\Ring_{n,n-1}$.

In Section \ref{SS_quant_p}, we reduce the $\Ring$-form from Section \ref{SS_quant_R}
mod $p$ for $p\gg 0$. We pass from the resulting microlocal quantization
$\D^\F_{n,n-1}$ to a ``Frobenius constant'' quantization $\underline{\D}_{n,n-1}^\F$
that turns out to be a splitting bundle for  the  Azumaya algebra $\left(\underline{\D}^{\F,opp}_{n}
\otimes \underline{\D}^{\F}_{n-1,\diamondsuit}\right)|_{H_{n,n-1}^{\F(1)}}.$

\subsection{Quantization in characteristic $0$}\label{SS_quant_0}
Let the base field be $\C$.
From (\ref{eq:period_opposite}) and  $\mathsf{Per}(\D_{n-1})=-\frac{1}{2}$, we get $\mathsf{Per}(\D_{n-1}^{\diamondsuit,opp})=\frac{1}{2}$. By Lemma \ref{Lem:nested_canonical},
for $n>2$,  we have
$$\iota^*(\mathsf{Per}(\D_n^{opp}\widehat{\otimes} \D_{n-1,\diamondsuit}))=-\frac{1}{2}c_1(K_Y),$$
where we write $\iota$ for the inclusion $H_{n,n-1}\hookrightarrow H_{n}\times H_{n-1,\diamondsuit}$.
It is lagrangian by Lemma \ref{Lem:lagrangian}.

The following lemma establishes the remaining assumption from Section \ref{SS_quant_lagrang}.

\begin{Lem}\label{Lem:nested_cohom_vanish}
We have $H^i(H_{n,n-1},\Str)=0$ for $i>0$ and $\C[H_{n,n-1}]=\C[V_n]^{S_{n-1}}$.
\end{Lem}
\begin{proof}
Recall the morphism $\rho_{n,n-1}:H_{n,n-1}\rightarrow V_n/S_{n-1}$, see Lemma \ref{Lem:nested_global}.
It is a birational proojective morphism. The variety $V_n/S_{n-1}$ has rational singularities.
It follows that $\rho_{n,n-1*}\Str_{n,n-1}=\Str_{V_n/S_{n-1}}$.
\end{proof}

Now using Corollary \ref{Cor:lag_quantization}  we see that there is a
filtered coherent $\D_n^{opp}\widehat{\otimes} \D_{n-1,\diamondsuit}$-module $\D_{n,n-1}$ with
a good filtration satisfying $\gr\D_{n,n-1}=\Str_{n,n-1}$. We can also assume that the
$T_h$-action lifts from $\Str_{n,n-1}$ to $\D_{n,n-1}$. Thanks to Lemma
\ref{Lem:nested_cohom_vanish}, we have a $T_h$-equivariant filtered isomorphism
$\gr\Gamma(\D_{n,n-1})=\C[V_n]^{S_{n-1}}$.

\begin{Prop}\label{Prop:glob_sections_nested}
We have an isomorphism $\Gamma(\D_{n,n-1})\cong \A_{n-1,\diamondsuit}$ of filtered $\A_{n-1,\diamondsuit}$-$\A_n$- bimodules.
\end{Prop}
\begin{proof}
Note that $\A_n$ embeds into $\A_{n-1,\diamondsuit}$ so we can view
$\Gamma(\D_{n,n-1})$ as an $\A_n$-bimodule. The filtration
0 component of $\Gamma(\D_{n,n-1})$ is one-dimensional thanks to the isomorphism
$\gr \Gamma(\D_{n,n-1})\cong \C[V_n]^{S_{n-1}}$.  We start by  proving  that this
component is  centralized by $\A_n$.

Let $a$ denote a nonzero element in the filtration component of degree $0$
in $\Gamma(\D_{n,n-1})$. This element is
$T_h$-invariant. The algebra $\A_n$  is generated by its subalgebras $\C[\underline{x}]^{S_n},
\C[\underline{y}]^{S_n}$, see, e.g., \cite{Wallach}, so it is enough to show that $a$ commutes with
$\C[\underline{x}]^{S_n}$ and $\C[\underline{y}]^{S_n}$.
We will do this for the first subalgebra, the second is analogous. Let $F\in \C[\underline{x}]^{S_n}$ be a homogeneous
element of degree $m$. Then it lies in the filtration degree $m$ component of $\A_n$
and also in the degree $m$ component for the grading induced by the $T_h$-action.
It follows that $[F,a]$ lies in the filtration component $m-2$ and in the grading component
$m$. However, since $\gr \Gamma(\D_{n,n-1})=\C[V_n]^{S_{n-1}}$, we see that the intersection of
that filtration component and that  grading component is zero. So $[F,a]=0$.

Consider the map
$\zeta: \A_{n-1,\diamondsuit}\rightarrow \Gamma(\D_{n,n-1}), b\mapsto ab$. The associated
graded morphism of $\zeta$ is an isomorphism of graded modules, so $\zeta$ is an isomorphism strictly compatible with
the filtrations. Since $a$ commutes with $\A_n$, we see that $\iota$ is an isomorphism
of bimodules.
\end{proof}

Now let us recover $\D_{n,n-1}$ (without the filtration) from the bimodule $\A_{n-1,\diamondsuit}$.

\begin{Prop}\label{Prop:localization}
We have the following isomorphism of $\D_{n-1,\diamondsuit}$-$\D_n$-bimodules
\begin{equation}\label{eq:localization_iso}\left(\D_{n-1,\diamondsuit}\widehat{\otimes}\D_{n}^{opp}\right)
\otimes_{(\A_{n-1,\diamondsuit}\otimes \A_n^{opp})}\A_{n-1,\diamondsuit}\xrightarrow{\sim}\D_{n,n-1}.\end{equation}
\end{Prop}
\begin{proof}
We claim that the functors $\Gamma$ and
$$\operatorname{Loc}(\bullet):=\left(\D_{n-1,\diamondsuit}\widehat{\otimes}\D_{n}^{opp}\right)
\otimes_{(\A_{n-1,\diamondsuit}\otimes \A_n^{opp})}(\bullet)$$
are mutually inverse equivalences between the category of coherent $\D_{n-1,\diamondsuit}\widehat{\otimes}\D_{n}^{opp}$-modules and the category of
finitely generated $\A_{n-1,\diamondsuit}\otimes \A_n^{opp}$-modules. This claim is usually
called an abelian localization theorem. This and Proposition
\ref{Prop:glob_sections_nested} imply (\ref{Prop:localization}).

Our claim that  $\Gamma,\Loc$ are mutually inverse equivalences is pretty classical and
can be proved in several different ways.
For example, it is a special case of the main result \cite{GS} (which is based upon Haiman's work so we cannot use that)
or of \cite{ab_loc} (which is independent from Haiman's work). Alternatively, it is a direct corollary
of the main results of \cite{MN} and \cite{MN_abelian}. Namely, the algebra
$\A_{n-1,\diamondsuit}\otimes \A_n^{opp}$ has finite homological dimension,
so $R\Gamma$ and $L\operatorname{Loc}$ are quasi-inverse derived equivalences
by \cite[Theorem 1.1]{MN}. By \cite[Corollary 1.3, Section 8]{MN_abelian}, the functor
$\Gamma$ is exact, so is an equivalence of abelian categories. Yet alternatively,
this follows from results of \cite{BEG} (that rational Cherednik algebras with integral parameters
are simple) and the general results on the abelian localization theorem
from \cite[Section 5.3]{BPW} and \cite[Section 4.2]{BL}.
\end{proof}

Now recall, Section \ref{SS_quant_lagrang}, that $\D_{n,n-1}$ is defined over $\Q$, let us denote the corresponding
$\Q$-form by $\D^\Q_{n,n-1}$. The proof of Proposition \ref{Prop:glob_sections_nested} shows
that $\Gamma(\D_{n,n-1}^{\Q})=\A_{n-1,\diamondsuit}^\Q$, an isomorphism of $\A^\Q_{n-1,\diamondsuit}$-$\A_n^\Q$-bimodules.
Proposition \ref{Prop:localization} also holds over $\Q$ because (\ref{eq:localization_iso}) is defined over $\Q$.

\begin{Rem}\label{Rem:quant_concrete_defn}
As we have seen, the varieties $H_{n},H_{n-1,\diamondsuit}$ as well as the quantizations $\D_n,\D_{n-1,\diamondsuit}$
can be constructed via (quantum) Hamiltonian reductions. The variety $H_{n,n-1}$ also has a ``Hamiltonian reduction''
description. It would be interesting to find such a description for $\D_{n,n-1}$.
\end{Rem}

\subsection{$\Ring$-form}\label{SS_quant_R}
As was mentioned in Section \ref{SS_Hilbert_general}, the schemes $H_n,H_{n-1}$ are defined over a finite localization $\Ring$ of $\Z$. We assume that $n!$
is invertible in $\Ring$ so taking the $S_n$-invariants behaves in the usual way.
The same is true for $H_{n,n-1}$.
Further localizing finitely many elements in $\Ring$, we can assume that the $\Ring$-schemes $H^\Ring_{n},H^\Ring_{n-1},
H^\Ring_{n,n-1}$ are regular and $R\Gamma(\Str_{i}^\Ring)=\Ring[V_i]^{S_i}$ for $i=n,n-1$, $R\Gamma(\Str^\Ring_{n,n-1})=\Ring[V_n]^{S_{n-1}}$.

As was mentioned in Section \ref{SS_quant_general}, we can further assume that the
microlocal filtered quantizations
$\D_n,\D_{n-1}$ are defined over $\Ring$ and that the corresponding $\Ring$-forms
$\D_n^\Ring, \D_{n-1}^\Ring$ are filtered microlocal quantizations of $\Str_n^\Ring,
\Str_{n-1}^\Ring$. By an $\Ring$-form (in the case of $\D_n$, for example),  we mean a subsheaf
$\D_n^\Ring\subset \D_n^{\Q}$ of $\Ring$-algebras such that the inclusion $\D_n^\Ring\hookrightarrow \D_n^{\Q}$ gives
an isomorphism  of filtered sheaves of algebras $\Q\widehat{\otimes}_\Ring\D_n^\Ring
\xrightarrow{\sim} \D_n^\Q$.

Our goal in this section is to prove the following technical result.

\begin{Lem}\label{Lem:R_form_quantization}
After replacing $\Ring$ with a finite localization, there is an $\Ring$-form $\D_{n,n-1}^\Ring$ of
$\D_{n,n-1}^{\Q}$ such that the following conditions hold:
\begin{enumerate}\item $\D_{n,n-1}^\Ring$ is a filtered $\D_{n-1,\diamondsuit}^{\Ring}$-$\D_n^\Ring$-bimodule
with respect to the filtration restricted from $\D_{n,n-1}^{\Q}$,
\item the filtration on $\D_{n,n-1}^\Ring$ is good,
\item and we have an isomorphism
$\gr \D_{n,n-1}^\Ring\xrightarrow{\sim} \Str_{n,n-1}^\Ring$ of graded sheaves (in the conical
topology) of bimodules.
\item $R\Gamma(\D_{n,n-1}^\Ring)=\A_{n-1,\diamondsuit}^\Ring$, the equality of
subbimodules of $R\Gamma(\D_{n,n-1}^\Q)=\A_{n-1,\diamondsuit}^\Q$.
\end{enumerate}
\end{Lem}
\begin{proof}
Consider the $\D_{n-1,\diamondsuit}^{\Ring}$-$\D_n^\Ring$-bimodule
$$\D_{n,n-1}'^\Ring:=(\D^{\Ring, opp}_n\widehat{\otimes}\D_{n-1,\diamondsuit}^{\Ring})\otimes_{\A_n^{\Ring,opp}\otimes
\A_{n-1,\diamondsuit}^{\Ring}}\A_{n-1,\diamondsuit}^\Ring.$$
The bimodule $\D_{n,n-1}'^\Ring$
admits a natural homomorphism of sheaves of bimodules, say $\iota$, to $\D_{n,n-1}^{\Q}$ and comes with the tensor
product filtration, which is good.
It follows from Proposition
\ref{Prop:localization} that $\iota$ induces an isomorphism
$\Q\widehat{\otimes}_\Ring\D'^\Ring_{n,n-1}\xrightarrow{\sim} \D^\Q_{n,n-1}$ of sheaves
of bimodules. Consider $\ker\iota\subset \D_{n,n-1}'^\Ring$. We see that $\gr\ker\iota$
is a coherent sheaf on $H_n^{\Ring}\times H_{n-1,\diamondsuit}^\Ring$ that is $\Ring$-torsion.
So after replacing $\Ring$ with a finite localization, $\iota$ becomes injective. Hence
$\D'^\Ring_{n,n-1}$ is an $\Ring$-form of $\D^Q_{n,n-1}$

Note that $\D_{n,n-1}'^{\Ring}$ also satisfies (1).
However,  it does not need to satisfy (3): even the filtration
on $\D_{n,n-1}^\Q$ induced from $\D_{n,n-1}'^\Ring$ does not need to coincide with
the initial filtration on $\D_{n,n-1}^\Q$. But both filtrations on $\D_{n,n-1}^{\Q}$ are good. We will show that
after replacing $\Ring$ with a finite localization there is a good filtration on
$\D'^\Ring_{n,n-1}$ satisfying (3). For this, consider the completed Rees sheaf $\D'^\Ring_{n,n-1,\hbar}$.
Then the completed Rees sheaf $\D'^\Q_{n,n-1,\hbar}$ is obtained as $\Q\widehat{\otimes}_\Ring
\D'^\Ring_{n,n-1,\hbar}$.

Recall, Lemma \ref{Lem:good_filt_compat},  that the good filtration on $\D_{n,n-1}^\Q$
is squeezed between appropriate shifts of the good filtration of
$\D'^\Q_{n,n-1}$.
So $\D^\Q_{n,n-1,\hbar}$ embeds into $\D'^\Q_{n,n-1,\hbar}$. The image contains
$\hbar^N \D'^{\Q}_{n,n-1,\hbar}$ for a sufficiently large integer $N$ and (after shifting the filtrations) we can assume
that the embedding is $T_c$-equivariant. Note that $\D'^\Q_{n,n-1,\hbar}/ \hbar^N
\D'^\Q_{n,n-1,\hbar}= \Q\otimes_\Ring \left(\D'^\Ring_{n,n-1,\hbar}/\hbar^N
\D'^\Ring_{n,n-1,\hbar}\right)$ (where we no longer need to complete the tensor product).
After replacing $\Ring$ with a finite localization, the subbimodule
$\D^\Q_{n,n-1,\hbar}/\hbar^N \D'^\Q_{n,n-1,\hbar}$ becomes defined over $\Ring$
and so gives rise to a $T_c$-stable subbimodule $\D^\Ring_{n,n-1,\hbar}\subset \D'^\Ring_{n,n-1,\hbar}$
containing $\hbar^N \D'^\Ring_{n,n-1,\hbar}$. From $\D^R_{n,n-1,\hbar}$ we produce
a filtered coherent $\D^{\Ring}_{n-1,\diamondsuit}$-$\D^\Ring_{n}$-bimodule $\D^\Ring_{n,n-1}$
with $\Q\widehat{\otimes}_\Ring \D^\Ring_{n,n-1}\xrightarrow{\sim} \D^\Q_{n,n-1}$.
Hence $\Q\otimes_{\Ring}\gr \D^\Ring_{n,n-1}\xrightarrow{\sim} \Str^{\Q}_{n,n-1}$.
After replacing $\Ring$ with a finite localization again, we achieve that
(3)  holds. Similarly, we achieve that (4) holds.
\end{proof}

\subsection{Quantization in characteristic $p$}\label{SS_quant_p}
Set $\F:=\overline{\F}_p$, where $\F_p$  is a quotient of $\Ring$
(so that $p$ is very large). Set $\D^\F_{n,n-1}:=\F\widehat{\otimes}_\Ring \D^\Ring_{n,n-1}$. This is a
microlocal filtered $\D^{\F}_{n-1,\diamondsuit}$-$\D_n^\F$-bimodule.

\begin{Lem}\label{Lem:microl_p_properties}
The bimodule $\D^\F_{n,n-1}$ is a microlocal filtered quantization of $\Str_{n,n-1}^\F$.
Further, we have $R\Gamma(\D^\F_{n,n-1})=\A_{n-1,\diamondsuit}^\F$, an isomorphism of
filtered bimodules.
\end{Lem}
\begin{proof}
This follows from Lemma \ref{Lem:R_form_quantization} because
$H_{n,n-1}^\Ring$ is flat over $\Spec(\Ring)$, while $\A_{n-1,\diamondsuit}^\Ring$
is flat over $\Ring$.
\end{proof}

Now we are going to produce a coherent sheaf $\underline{\D}^\F_{n,n-1}$ on $H_{n}^{\F(1)}\times
H_{n-1,\diamondsuit}^{\F(1)}$.  Let $\operatorname{Fr}$ denote the Frobenius morphism for
$H_n\times H_{n-1,\diamondsuit}$. Namely, consider the completed Rees
sheaf $\D_{n,n-1,\hbar}^\F$, this is a coherent $\D_{n-1,\diamondsuit,\hbar}^\F$-$\D_{n,\hbar}^\F$-bimodule
(and a sheaf in the Zariski topology on $H_{n}^\F\times H_{n-1,\diamondsuit}^\F$).
Note that $\operatorname{Fr}_*\left(\D_{n-1,\diamondsuit,\hbar}^\F\widehat{\otimes}_{\F[[\hbar]]}
\D_{n,\hbar}^{\F,opp}\right)$ is a coherent sheaf of  $\Str_{H_{n}^{\F(1)}\times H_{n-1,\diamondsuit}^{\F(1)}}[[\hbar]]$-modules.
It follows that $\operatorname{Fr}_*(\D_{n,n-1,\hbar}^\F)$ is also a coherent sheaf of
$\Str_{H_{n}^{\F(1)}\times H_{n-1,\diamondsuit}^{\F(1)}}[[\hbar]]$-modules.

Now note that, by the construction, $\operatorname{Fr}_*(\D_{n,n-1,\hbar}^\F)$  is $T_c^\F$-equivariant.
Since the action of $T_c^\F$-contracting, the functor of  $\hbar$-adic completion is
a category equivalence
$$\Coh^{T_c}(H_{n}^{\F(1)}\times
H_{n-1,\diamondsuit}^{\F(1)}\times \mathbb{A}^1)\xrightarrow{\sim}
\Coh^{T_c}(\Str_{H_{n}^{\F(1)}\times H_{n-1,\diamondsuit}^{\F(1)}}[[\hbar]]).$$
Let $\underline{\D}_{n,n-1,\hbar}^\F$
denote the $T_c^\F$-equivariant coherent sheaf on $H_{n}^{\F(1)}\times
H_{n-1,\diamondsuit}^{\F(1)}\times \mathbb{A}^1$ corresponding to $\operatorname{Fr}_*(\D_{n,n-1,\hbar}^\F)$.
Set $\underline{\D}_{n,n-1}^\F:=\underline{\D}_{n,n-1,\hbar}^\F/(\hbar-1)$. This is a coherent sheaf
on $H_{n}^{\F(1)}\times H_{n-1,\diamondsuit}^{\F(1)}$. Note that its restriction to the conical
topology admits a natural embedding into $\operatorname{Fr}_*\D^\F_{n,n-1}$.

Let us record some basic properties of $\underline{\D}_{n,n-1}^\F, \underline{\D}_{n,n-1,\hbar}^\F$.

\begin{Lem}\label{Lem:underline_D_properties}
The following claims are true:
\begin{enumerate}
\item The sheaf $\underline{\D}_{n,n-1,\hbar}^\F$ is flat over $\mathbb{A}^1$ and
the specialization to $\hbar=0$ equals $\operatorname{Fr}_* \Str^\F_{n,n-1}$.
\item The sheaf $\underline{\D}_{n,n-1}^\F$ carries a natural structure of
a $\underline{\D}_{n-1,\diamondsuit}^\F$-$\D_n^\F$-bimodule.
\item We have $R\Gamma(\underline{\D}_{n-1,\diamondsuit}^\F)\cong \A_{n-1,\diamondsuit}^\F$,
an isomorphism of $\A_{n-1,\diamondsuit}^\F$-$\A_{n}^\F$-bimodules.
\end{enumerate}
\end{Lem}
\begin{proof}
(1) follows directly from the constructions of $\underline{\D}^\F_{n,n-1,\hbar},
\underline{\D}^\F_{n,n-1}$. To prove (2), notice that the construction
that produces $\underline{\D}^\F_{n,n-1}$ from $\D^\F_{n,n-1}$ also applies to
$\D^\F_n, \D^\F_{n-1,\diamondsuit}$ and produces $\underline{\D}^\F_n,\underline{\D}^\F_{n-1,\diamondsuit}$.
(3) follows from the construction of $\underline{\D}^\F_{n,n-1}$ and the analogous
properties of $\D^\F_{n,n-1}$ in Lemma \ref{Lem:microl_p_properties}.
\end{proof}

It turns out that $\underline{\D}^\F_{n,n-1}$ is a ``Frobenius constant'' quantization of
$\Str_{n,n-1}^\F$ in the sense of the following lemma.

\begin{Lem}\label{Lem:Frob_const_bimodule}
The action of $\Str_{H_n^{\F(1)}\times H_{n-1}^{\diamondsuit\F(1)}}$
on $\underline{\D}^\F_{n,n-1}$ factors through $\Str_{n,n-1}^{\F(1)}$.
\end{Lem}
\begin{proof}
Since $\underline{\D}^{\F}_{n,n-1}$ embeds into $\operatorname{Fr}_* \D_{n,n-1}^\F$,
it is enough to prove our claim for the latter sheaf.

Let $\iota$ denote the inclusion $V^{0\F(1)}/S_n\times V^{0\F(1)}/S_{n-1}
\hookrightarrow H_{n}^{\F(1)}\times H_{n-1,\diamondsuit}^{\F(1)}$. Note that
the pull-back $\iota^* \operatorname{Fr}_* \D^\F_{n,n-1}$
coincides with the microlocalization of the bimodule $\Gamma(\D_{n,n-1}^\F)=\A_{n-1,\diamondsuit}^\F$
to $V_n^{0\F(1)}/S_n\times V^{0\F(1)}/S_{n-1}$ (that can also be viewed as an open subset of
$V_n^{\F(1)}/S_{n}\times V_n^{\F(1)}/S_{n-1}$). It is easy to see that
the $\F[V_n^{\F(1)}]^{S_{n-1}}$-$\F[V_n^{\F(1)}]^{S_n}$-bimodule $\A_{n-1,\diamondsuit}^\F$
is scheme theoretically supported on the diagonal  $V_n^{\F(1)}/S_{n-1}
\hookrightarrow V_n^{\F(1)}/S_{n}\times V_n^{\F(1)}/S_{n-1}$.
Hence  $\iota^* \operatorname{Fr}_* \D^\F_{n,n-1}$
is scheme theoretically supported on $H_{n,n-1}^{\F(1)}\cap (V_n^{0\F(1)}/S_n\times
V_n^{0\F(1)}/S_{n-1})$.

Now note that
$\operatorname{Fr}_*\Str_{H_{n,n-1}^\F}\subset
\iota_*\iota^*\operatorname{Fr}_*\Str_{H_{n,n-1}^\F}.$
Since $\operatorname{Fr}_*\D^\F_{n,n-1}$ is a microlocal filtered quantization
of $\operatorname{Fr}_*\Str_{H_{n,n-1}^\F}$, it follows that
$$\operatorname{Fr}_*\D_{n,n-1}^\F\subset
\iota_*\iota^*\operatorname{Fr}_*\D_{n,n-1}^\F.$$
Since $\iota^* \operatorname{Fr}_* \D^\F_{n,n-1}$ is scheme theoretically
supported on $H_{n,n-1}^{(1)\F}\cap (V_n^{0\F(1)}/S_n\times
V_n^{0\F(1)}/S_{n-1})$, we deduce that $\operatorname{Fr}_*\D_{n,n-1}^\F$
is scheme theoretically supported on $H_{n,n-1}^{\F(1)}$, which is what we
need to prove.
\end{proof}

Here is the final property of $\underline{\D}_{n,n-1}^\F$ we need.

\begin{Lem}\label{Lem:splitting}
The sheaf $\underline{\D}^\F_{n,n-1}$ on $H_{n,n-1}^{\F(1)}$ is a splitting bundle for the Azumaya
algebra
\begin{equation}\label{eq:Azumaya}\left(\underline{\D}^{\F}_{n-1,\diamondsuit}\otimes \underline{\D}^{\F,opp}_n)\right)|_{H_{n,n-1}^{\F(1)}}.\end{equation}
\end{Lem}
\begin{proof}
As was mentioned in Section \ref{SS_quant_general}, the rank of the Azumaya algebra $\underline{\D}^{\F}_n$
is equal to $p^{2n}$. Similarly, the rank of $\underline{\D}^\F_{n-1,\diamondsuit}$
By (2) of Lemma \ref{Lem:underline_D_properties}, $\underline{\D}_{n,n-1}^\F$
is a module over (\ref{eq:Azumaya}). So the rank of every fiber of
$\underline{\D}_{n,n-1}^\F$ is $\geqslant p^{2n}$. What we need to prove is that
the rank of every fiber is $p^{2n}$. It  is enough to prove the analogous claim for
$\underline{\D}_{n,n-1,\hbar}^\F$ viewed as a coherent sheaf on $H_{n,n-1}^{\F(1)}\times
\mathbb{A}^1$. Recall that this sheaf is $T_c^\F$-equivariant and that the action is
contracting. By (1) of Lemma \ref{Lem:underline_D_properties}, the rank of
$\underline{\D}^\F_{n,n-1,\hbar}$ on $H_{n,n-1}^{\F(1)}\times \{0\}$ is $p^{2n}$.
By the semi-continuity of rank, the rank at any other point of
$H_{n,n-1}^{\F(1)}\times\mathbb{A}^1$ is $\leqslant p^{2n}$. This completes the proof.
\end{proof}

\section{Proof of (\ref{eq:BK_iso_Pro}) in characteristic $p\gg 0$}
The goal of this section is to establish an $\tilde{A}_{n-1,\diamondsuit}$-linear
isomorphism $\alpha_{n*}\beta_n^*\Pro_{n-1,\diamondsuit}^\F\xrightarrow{\sim} \Pro_n^\F$.
The proof is in two steps.
First, we use results of Section \ref{SS_quant_p} together with the
construction of the Procesi bundles from \cite{BK} recalled in Section \ref{SS_Procesi_constr}
to prove a weaker version of (\ref{eq:BK_iso_Pro}),
where we twist  by some line bundles. Then we  prove that no nontrivial twists can
occur.

\subsection{Weaker version of (\ref{eq:BK_iso_Pro})}
Our goal in this section is to prove the following statement. We use the
notation from Section \ref{SS_Procesi_constr}.

\begin{Prop}\label{Prop:commut_diagram}
There are integers $k,\ell$ such that there exists an
$\widetilde{A}_{n-1,\diamondsuit}^{\F}$-linear isomorphism
\begin{equation}\label{eq:BK_iso_weak}
\alpha_{n*}\beta_n^* \Pro_{n-1,\diamondsuit}^\F(k)\cong \Pro_{n}^\F(\ell).
\end{equation}
%
\end{Prop}
\begin{proof}
{\it Step 1}. First of all, we can replace $H_n,V_n$ etc. with their Frobenius twists.
We claim that it is also enough to prove the completed version of (\ref{eq:BK_iso_weak}), where we  replace
$H_n^{\F(1)}$ with $H_{n}^{\F(1)\wedge}:=\rho^{-1}(V_n^{\F(1)\wedge}/S_n)$, etc..
So suppose we know $\alpha_{n*}\beta_n^* \hat{\Pro}_{n-1,\diamondsuit}^\F(k)\cong \hat{\Pro}_{n}^\F(\ell)$
and we want to deduce (\ref{eq:BK_iso_weak}).

Recall from Section \ref{SS_Procesi_constr} that the bundle $\hat{\Pro}_n^{\F}$ is rigid and admits a
$T_c^{\F(1)}$-equivariant structure. As was mentioned in Section \ref{SS_Procesi_constr},
each equivariant structure gives rise to an extension of
$\hat{\Pro}_{n}^{\F}$ to  $H_n^{\F(1)}$ but the extension as a vector bundle is independent
of the choice of an equivariant structure. Now take a $T_c^{\F(1)}$-equivariant structure
on $\hat{\Pro}_{n-1,\diamondsuit}^\F$. It gives rise to an equivariant structure on
$\hat{\Pro}_{n-1,\diamondsuit}^{\F}(k)$ and, since $\alpha_n,\beta_n$ are equivariant,
also  on $\hat{\Pro}_n^{\F}(\ell)$. The latter comes from the isomorphism
$\hat{\Pro}_n^{\F}(\ell)\cong \alpha_{n*}\circ \beta_n^*(\hat{\Pro}_{n-1,\diamondsuit}^{\F}(k))$.
This equivariant structure on $\hat{\Pro}_n^{\F}(\ell)$   will give rise to (\ref{eq:BK_iso_weak}).

{\it Step 2}.
Consider the following diagram.
\begin{equation}\label{eq:diagr_induction2}
\begin{picture}(120,35)
\put(2,2){$D^b(\A_{n-1,\diamondsuit}^{\F opp}\operatorname{-mod})$}
\put(82,2){$D^b(\A_n^{\F opp}\operatorname{-mod})$}
\put(2,22){$D^b(\Coh(\underline{\D}_{n-1,\diamondsuit}^{\F,opp}))$}
\put(82,22){$D^b(\Coh(\underline{\D}_n^{\F,opp})$}
\put(18,21){\vector(0,-1){14}}
\put(93,21){\vector(0,-1){14}}
\put(34,3){\vector(1,0){47}}
\put(34,23){\vector(1,0){47}}
\put(35,26){\tiny $\alpha^{(1)}_{n*}\left(\underline{\D}_{n,n-1}^\F\otimes_{\underline{\D}_{n-1,\diamondsuit}^{\F,opp}}\beta_n^{(1)*}(\bullet)\right)$}
\put(60,4){\tiny $\operatorname{Res}$}
\put(19,12){\tiny $R\Gamma$}
\put(94,12){\tiny $R\Gamma$}
\end{picture}
\end{equation}
Here $\operatorname{Res}$ is the restriction functor for the inclusion $\A^\F_{n}\hookrightarrow \A^\F_{n-1,\diamondsuit}$.
Note that the vertical arrows are equivalences by \cite[Proposition 2.2]{BK}. Also note that the diagram
is commutative. Indeed, the composition
$$D^b(\A_{n-1,\diamondsuit}^{\F opp}\operatorname{-mod})\rightarrow D^b(\Coh(\D_{n-1,\diamondsuit}^{\F,opp}))
\rightarrow D^b(\Coh(\D_n^{\F,opp})\rightarrow D^b(\A_n^{\F opp}\operatorname{-mod})$$
is the derived tensor product with $R\Gamma(\underline{\D}^\F_{n,n-1})$. By (3) of Lemma
\ref{Lem:underline_D_properties}, the latter object is $\A_{n-1,\diamondsuit}^\F$.
The commutativity follows.

(\ref{eq:diagr_induction2}) remains commutative after we restrict it to $V_n^{\F(1)\wedge}$. So we  get the following commutative diagram.
\begin{equation}\label{eq:diagr_induction3}
\begin{picture}(120,35)
\put(2,2){$D^b(\A_{n-1,\diamondsuit}^{\F\wedge,opp}\operatorname{-mod})$}
\put(82,2){$D^b(\A_n^{\F\wedge,opp}\operatorname{-mod})$}
\put(6,22){$D^b(\Coh(\underline{\D}_{n-1,\diamondsuit}^{\F\wedge,opp}))$}
\put(82,22){$D^b(\Coh(\underline{\D}_n^{\F\wedge,opp})$}
\put(18,21){\vector(0,-1){14}}
\put(93,21){\vector(0,-1){14}}
\put(34,3){\vector(1,0){48}}
\put(39,23){\vector(1,0){42}}
\end{picture}
\end{equation}

{\it Step 3}. Consider the vector bundle $\underline{\D}_{n,n-1}^{\F\wedge}$ on $H_{n,n-1}^{\F(1)\wedge}$.
By Lemma \ref{Lem:splitting}, it is a splitting bundle
for the Azumaya algebra $\alpha_{n}^{(1)*}\underline{\D}_{n}^{\F\wedge,opp}\otimes \beta_{n}^{(1)*}\underline{\D}_{n-1,\diamondsuit}^{\F\wedge}
$. But both Azumaya algebras $\underline{\D}_{n}^{\F\wedge,opp}$ and $\underline{\D}_{n-1,\diamondsuit}^{\F\wedge}$
split with splitting  bundles  $\hat{\E}_n^{\F*}, \hat{\E}_{n-1,\diamondsuit}^\F$, respectively.
Since a splitting bundle is defined uniquely up to a twist with a line bundle,
there is a line bundle $\mathcal{L}$ on $H_{n,n-1}^{\F(1)\wedge}$ such that we have
the following isomorphism:
\begin{equation}\label{eq:compl_iso}
\underline{\D}_{n,n-1}^{\F\wedge}\cong \alpha_{n}^{(1)*}\left(\hat{\E}_{n}^{\F*}\right)\otimes \beta_{n}^{(1)*}(\hat{\E}_{n-1,\diamondsuit}^{\F})
\otimes \mathcal{L}.
\end{equation}

We claim that $\mathcal{L}=\Str(-k,\ell)^{\wedge}$ for some $k,\ell\in \Z$. Indeed, the higher cohomology
of $\Str_{H_{n,n-1}^{\F(1)\wedge}}$ vanishes because the same holds for $H_{n,n-1}^\F$. So every line
bundle on $H_{n,n-1}^{\F(1)\wedge}$ is rigid. Therefore it extends to $H_{n,n-1}^{\F(1)}$. Our claim
that $\mathcal{L}=\Str(-k,\ell)^{\wedge}$ follows from Lemma \ref{Lem:Pic_nested}. So  we see that
\begin{equation}\label{eq:compl_iso1}
\D_{n,n-1}^{\F\wedge}\cong \alpha_{n}^{(1)*}(\hat{\E}_{n}(k)^{\F*})\otimes \beta_{n}^{(1)*}(\hat{\E}_{n-1,\diamondsuit}(\ell)^{\F}).
\end{equation}

{\it Step 4}. The following commutative diagram is a consequence of (\ref{eq:compl_iso1}) and
commutative diagram (\ref{eq:diagr_induction3}):
\begin{equation}\label{eq:diagr_induction1}
\begin{picture}(120,30)
\put(2,2){$D^b(\A_{n-1,\diamondsuit}^{\F\wedge,opp}\operatorname{-mod})$}
\put(82,2){$D^b(\A_n^{\F\wedge,opp}\operatorname{-mod})$}
\put(10,22){$D^b(\Coh(H^{\F\wedge}_{n-1,\diamondsuit}))$}
\put(82,22){$D^b(\Coh(H^{\F\wedge}_n))$}
\put(23,21){\vector(0,-1){14}}
\put(93,21){\vector(0,-1){14}}
\put(34,3){\vector(1,0){47}}
\put(42,23){\vector(1,0){40}}
\put(60,24){\tiny $\alpha_{n*}\circ \beta_n^*$}
\put(24,12){\tiny $R\operatorname{Hom}(\hat{\E}^{\F}_{n-1,\diamondsuit}(k),\bullet)$}
\put(94,12){\tiny $R\Hom(\hat{\E}^{\F}_n(\ell),\bullet)$}
\end{picture}
\end{equation}

{\it Step 5}. Let $\hat{\B}^\F_n,\hat{\B}^{o,\F}_n$ have the same meaning as in Section and let
$\hat{\B}_{n-1,\diamondsuit}^\F, \hat{\B}_{n-1,\diamondsuit}^{o,\F}$ be defined similarly. We have
\begin{align*}
&\hat{\B}^\F_n\otimes_{\A_n^{\F\wedge}}\hat{\B}^{o,\F}_n=\widetilde{A}_n^{\F(1)\wedge}, \quad
\hat{\B}^\F_{n-1,\diamondsuit}\otimes_{\A_{n-1,\diamondsuit}^{\F\wedge}}
\hat{\B}^{o,\F}_{n-1,\diamondsuit}=\widetilde{A}_{n-1,\diamondsuit}^{\F(1)\wedge},\\
&\hat{\Pro}_{n}^{\F}=\hat{\B}^\F_n\otimes_{\A_n^{\F\wedge}} \hat{\E}^\F_n,
\hat{\Pro}_{n-1,\diamondsuit}^{\F}=
\hat{\B}^\F_{n-1,\diamondsuit}\otimes_{\A_{n-1,\diamondsuit}^{\F\wedge}} \hat{\E}_{n-1,\diamondsuit}^{\F},\\
&\hat{\B}^\F_{n-1,\diamondsuit}\otimes_{\A_{n-1,\diamondsuit}^{\F\wedge}}\D_{n,n-1}^{\F\wedge}
\otimes_{\A_n^{\F\wedge}}\hat{\B}^{o,\F}_n
\cong \alpha_{n}^{(1)*}(\hat{\Pro}_{n}(k)^{\F})^*\otimes \beta_{n}^{(1)*}(\hat{\Pro}_{n-1,\diamondsuit}(\ell)^{\F}).
\end{align*}
The first two lines are from Section \ref{SS_Procesi_constr}, while the third line
follows from the first two and (\ref{eq:compl_iso1}).

We claim that we have the following isomorphism of $\widetilde{A}_{n-1,\diamondsuit}^{\F(1)\wedge}$-$\widetilde{A}_{n}^{\F(1)\wedge}$-bimodules
\begin{equation}\label{eq:bimod_iso}
\hat{\B}^\F_{n-1,\diamondsuit}\otimes_{\A_{n-1,\diamondsuit}^{\F\wedge}}\A_{n-1,\diamondsuit}^{\F\wedge}
\otimes_{\A_{n}^{\F\wedge}}\hat{\B}^{o,\F}_n\cong \widetilde{A}_{n}^{\F(1)\wedge}.
\end{equation}
Indeed, recall from Section \ref{SS_Procesi_constr} that the Morita equivalence between $\A_{n}^{\F\wedge}$
and $\widetilde{A}_{n}^{\F(1)\wedge}$ is composed of two Morita equivalences:
\begin{itemize}
\item[(a)]  $\Weyl(V_n^\F)^{\wedge}\otimes_{\A_{n}^{\F\wedge}}\bullet: \A_n^{\F\wedge}\operatorname{-mod}\xrightarrow{\sim}\Weyl(V_n^\F)^{\wedge}\operatorname{-mod}^{S_n}
$,
\item[(b)] the Morita equivalence $\Weyl(V_n^\F)^{\wedge}\operatorname{-mod}^{S_n}
\xrightarrow{\sim} \F[V_n^{\F(1)}]^{\wedge}\operatorname{-mod}^{S_n}$ given by the
standard splitting module $\F[\underline{x}^p][[\underline{y}]]$ for $\Weyl(V_n^\F)^{\wedge}$, which is $S_n$-equivariant.
\end{itemize}
The Morita equivalence for $n-1$ is defined similarly, where in (b) we use the same
standard splitting module. So applying the equivalences in (a) on the left
and on the right to $\A_{n-1,\diamondsuit}^{\F\wedge}$ we get
\begin{align*}\Weyl(V_n^\F)^\wedge\otimes_{\A_{n-1,\diamondsuit}^{\F\wedge}}\A_{n-1,\diamondsuit}^{\F\wedge}
\otimes_{\A_n^{\F\wedge}}\Weyl(V_n^\F)^\wedge
\cong \Weyl(V_n^\F)^\wedge
\otimes_{\A_n^{\F\wedge}}\Weyl(V_n^\F)^\wedge\cong \Weyl(V_n^{\F})^{\wedge}\# S_{n}\end{align*}
And then applying the equivalences in (b), we get $\widetilde{A}_{n}^{\F(1)\wedge}$.
This proves (\ref{eq:bimod_iso}).

{\it Step 6}. Applying the Morita equivalences from Step 5 to (\ref{eq:diagr_induction1})

\begin{equation}\label{eq:diagr_induction4}
\begin{picture}(120,30)
\put(10,2){$D^b(\widetilde{A}^{\F(1)\wedge}_{n-1,\diamondsuit}\operatorname{-mod})$}
\put(82,2){$D^b(\widetilde{A}^{\F(1)\wedge}_n\operatorname{-mod})$}
\put(10,22){$D^b(\Coh(H^{\F(1)\wedge}_{n-1,\diamondsuit}))$}
\put(82,22){$D^b(\Coh(H^{\F(1)\wedge}_n))$}
\put(23,21){\vector(0,-1){14}}
\put(93,21){\vector(0,-1){14}}
\put(42,3){\vector(1,0){40}}
\put(42,23){\vector(1,0){40}}
\put(60,24){\tiny $\alpha_{n*}\circ \beta_n^*$}
\put(60,4){\tiny $\operatorname{Ind}$}
\put(24,12){\tiny $R\operatorname{Hom}(\hat{\Pro}^{\F}_{n-1,\diamondsuit}(k),\bullet)$}
\put(94,12){\tiny $R\Hom(\hat{\Pro}^{\F}_n(\ell),\bullet)$}
\end{picture}
\end{equation}
As in Step 2, the vertical arrows are equivalences.
Computing the image of $\widetilde{A}^{\F(1)\wedge}_{n-1,\diamondsuit}$ in
$D^b(\Coh(H^{\F(1)\wedge}_n))$ using diagram (\ref{eq:diagr_induction4})
in two different ways we arrive at $\alpha_{n*}\beta_n^*(\hat{\Pro}^\F_{n-1,\diamondsuit}(\ell))
\cong \hat{\Pro}^\F_{n}(k)$, an $\widetilde{A}_{n-1,\diamondsuit}^{\F(1)\wedge}$-linear isomorphism.
Thanks to Step 1, the proof of the proposition is complete.
\end{proof}

%
%
%

\subsection{Absence of twists}
Here we prove the following lemma.

\begin{Lem}\label{Lem:no_twist}
In Proposition \ref{Prop:commut_diagram} we have $k=\ell=0$.
\end{Lem}
\begin{proof}
Recall that $(\Pro^\F_{n-1,\diamondsuit})^{S_{n-1}}\cong \Str_{n-1,\diamondsuit}^\F$
by the construction. So (\ref{eq:BK_iso_weak}) implies
\begin{equation}\label{eq:iso_pushforward}
\alpha_{n*}(\Str_{n,n-1}^\F(\ell,0))\cong \Pro_n^\F(k)^{S_{n-1}}.\end{equation}
Consider the restrictions of both sides of (\ref{eq:iso_pushforward})
to $\rho_n^{-1}((V_n^\F/S_n)^{\wedge_b})$, where $b$ is as in Lemma
\ref{Lem:Hilb_local}.

By Proposition \ref{Prop:Procesi_classif}, the right hand side of
(\ref{eq:iso_pushforward}) is either $\mathcal{T}_n^\F(k)$ or $\mathcal{T}_n^{\F^*}(k)$.
So by Lemma \ref{Lem:taut_restriction}, the restriction of the right hand side is
\begin{equation}\label{eq:res1}\Str^{\F\wedge}(k)^{\oplus n-1}\oplus \Str^{\F\wedge}(k\pm 1),\end{equation}
where the plus sign corresponds
to $\mathcal{T}_n^\F$ and the minus sign corresponds to $\mathcal{T}_n^{\F*}$.

The restriction of $\Str_{n,n-1}^\F(\ell,0)$
to $\alpha_n^{-1}(\rho^{-1}((V_n^\F/S_n)^{\wedge_b}))$ is $\Str^{\F\wedge}_{n,n-1}$ for the components of form
(i) (in Lemma \ref{Lem:nested_Hilb_local}) and $\Str^{\F\wedge}_{n,n-1}(\ell)$ for the component of form (ii).
Let us compute $\alpha_{2*}(\Str^{\F\wedge}_{2,1}(\ell))$. First, note that $\alpha^*_2\Str_2^\F(1)=\Str_{2,1}^\F(2)$.
Thanks to the projection formula, it is enough to compute $\alpha_{2*}(\Str^{\F\wedge}_{2,1}(\ell))$ for  $\ell=0,-1$.
We claim that here $\alpha_{2*}\Str_{2,1}^\F(\ell)\cong \Str_2^\F(\ell)\oplus \Str_2^\F(\ell+1)$.
Indeed, we already know this for $\ell=0$, while the bundle $\Str_{2,1}^\F(-1)$ is the canonical bundle
of $H^\F_{n,n-1}$ so  $\alpha_{2*}\Str_{2,1}^\F(-1)=(\alpha_{2*}\Str_{2,1}(0))^*$.

We conclude that, for $\ell=2\ell_1+\ell_0$ with $\ell_0\in \{-1,0\}$, the restriction of the left hand side of (\ref{eq:iso_pushforward})
is
\begin{equation}\label{eq:res2}(\Str_2^{\F\wedge})^{\oplus n-2}\oplus \Str_2^{\F\wedge}(\ell_1+\ell_0)+\Str_2^{\F\wedge}(\ell_1+\ell_0+1).\end{equation}

Comparing (\ref{eq:res1}) and (\ref{eq:res2}), we see that we have the following
options:
\begin{itemize}
\item[(I)] $k=\ell=0$.
\item[(II)] $k=0,\ell=-1$.
\item[(III)] $n=3, k=\ell=1$.
\end{itemize}
In particular, $\Pro_2^\F=\Str^\F\oplus \Str^\F(1)$ (here $\ell=0$ by default).

Let us show that (III) is not possible. Here we have $(\Pro_3^\F)^{S_2}\cong
\mathcal{T}_3^\F$. From Lemma \ref{Lem:Procesi_sgn} it follows  that $\Str_3^\F(2)=\left(\Pro_3^\F(1)\right)^{\sgn}$.
Note that $$\alpha_{3*}\circ \beta_3^*(\Str^\F(2))=\alpha_{3*}\circ \beta_3^*(\Pro_2^\F(1)^{\sgn})=(\Pro_3^\F(1))^{S_2,\operatorname{sgn}}$$
By the computations above, the restriction of the left hand side to $\rho^{-1}((V_3^\F/S_3)^{\wedge_b})$
is $\Str^\F(0)\oplus \Str^\F(1)\oplus \Str^\F(2)$. However, the restriction of $(\Pro_3^\F(1))^{S_2,\operatorname{sgn}}$
may only have two pairwise nonisomorphic summands, compare to the proof of Lemma
\ref{Lem:Procesi_sgn}.

Let us show that (II) is not possible. We induct on $n$ with the base of $n=2$ done above.
In particular, $(\Pro_{n-1}^{\F})^{S_{n-2}}=\mathcal{T}_{n-1}^\F$ and hence,
by Lemma \ref{Lem:Procesi_sgn}, $(\Pro_{n-1}^\F)^{\operatorname{sgn}}=\Str^\F(1)$.
On the other hand, $(\Pro^\F_n)^{S_{n-1}}=\mathcal{T}_n^{\F*}$. From (\ref{eq:BK_iso_weak}) we deduce that
\begin{equation}\label{eq:iso_pushforward1}(\Pro^\F_n)^{S_{n-1},\sgn}=\alpha_{n*}\Str_{n,n-1}^\F.
\end{equation} But
$(\Pro^\F_n)^{S_{n-1},\sgn}=(\Pro^\F_n)^{\sgn}\otimes \left((\Pro^\F_n)^{S_{n-1}}\right)^*=
\mathcal{T}_n^\F(-1)$, the former equality follows from Corollary
\ref{Cor:Procesi_reln}, while the latter   holds thanks to Lemma \ref{Lem:Procesi_sgn}.
Restricting both sides of (\ref{eq:iso_pushforward1}) to $\rho_n^{-1}((V_n^\F/S_n)^{\wedge_b})$
we see that the restrictions differ by a twist with $\Str(-1)$, which leads to a contradiction.
\end{proof}

\section{On normalization of  $IH_n$}\label{S_n!_proof}
\subsection{Theorem \ref{Thm:bijective_normalization} in characteristic $p$}
\begin{Prop}\label{Prop:bijective_normal_p}
The following claims are true:
\begin{enumerate}
\item The sheaf $\Pro_{n}^\F$ has a natural algebra structure, let $X_n^\F$ denote
the relative spectrum of $\Pro_n^\F$.
\item  The variety $X_n^\F$ is normal, Cohen-Macaulay and Gorenstein.
The canonical bundle is obtained by pulling back $\Str^\F_n(-1)$ from $H_n^\F$.
\item The natural morphism $X_n^\F\rightarrow IH_n^\F$ is bijective
and is an isomorphism over $V_n^{1\F}/S_n$.
\end{enumerate}
\end{Prop}
\begin{proof}
We will prove this by induction on $n$ starting with $n=2$. Here $IH^\F_2=U^\F_2=H^\F_{2,1}$
are Cohen-Macaulay and Gorenstein and (1)-(3) hold. Now suppose that they hold for
$n-1$ and prove them for $n$.

{\it Step 1}. Let us construct a sheaf of algebras structure on $\Pro_n^\F$.
Set $X^\F_{n,n-1}:=H^\F_{n,n-1}\times_{H^\F_{n-1}}X^\F_{n-1}$. This  scheme is finite and flat over
$H^\F_{n,n-1}$ so it is Cohen-Macaulay. Over $V_{n-1,\diamondsuit}^{0\F}/S_{n-1}\subset H_{n,n-1}^\F$,
the scheme $X^\F_{n,n-1}$ becomes $V^{0,\F}_{n-1,\diamondsuit}$. Therefore $X^\F_{n,n-1}$
is generically reduced, hence reduced. We denote the natural morphism $X^\F_{n,n-1}\rightarrow H^\F_{n,n-1}$ by
$\pi_{n,n-1}$ and the morphism $X^\F_{n,n-1}\rightarrow X^\F_{n-1,\diamondsuit}$
by $\widetilde{\beta}_{n}$. 

Consider the composition $\alpha_n\circ \pi_{n,n-1}:X^\F_{n,n-1}\rightarrow H^\F_n$.
Note that  $(\alpha_n\circ \pi_{n,n-1})_* \Str_{X^\F_{n,n-1}}=\alpha_{n*}\beta_n^* \Pro_{n-1,\diamondsuit}^\F$.
Hence, by Proposition \ref{Prop:commut_diagram} combined with Lemma \ref{Lem:no_twist},
\begin{equation}\label{eq:X_pushforward}(\alpha_n\circ \pi_{n,n-1})_* \Str_{X^\F_{n,n-1}}\cong\Pro_n^\F,
\end{equation} an $\widetilde{A}_{n-1,\diamondsuit}^\F$-linear isomorphism.
This isomorphism, in particular, equips $\Pro_n^\F$ with the structure of a sheaf of algebras.

{\it Step 2}.   Set $X_n^\F:=\Spec_{\Str_n^\F}(\Pro_n^\F)$. By the
construction of $X_n^\F$ we have a natural morphism $\widetilde{\alpha}_n:X_{n,n-1}^\F
\rightarrow X_n^\F$. It satisfies
\begin{equation}\label{eq:X_alpha_morphism}\widetilde{\alpha}_{n*}\Str_{X_{n,n-1}^\F}=\Str_{X_n^\F}.\end{equation}
So we have a degree $n!$ finite morphism  $X_n^\F\rightarrow H_{n}^\F$ to be denoted by
$\pi_n$, it satisfies $\alpha_n\circ \pi_{n,n-1}=\pi_n\circ \widetilde{\alpha}_n$.
In particular, $X_n^\F$ is Cohen-Macaulay. Over $V_n^{0,\F}/S_n\subset H_n^\F$,
the morphism $\alpha_n\circ \pi_{n,n-1}$ becomes the natural morphism
$V^{0\F}_n\rightarrow V^{0\F}_n/S_n$. Therefore $X_n^\F\times_{H_n}(V_n^{0\F}/S_n)=V_n^{0\F}$.
It follows that $X_n^\F$ is reduced.

{\it Step 3}. Let us construct a morphism $X_n^\F\rightarrow IH_n^\F$. Note that $X_{n,n-1}^\F$
comes with a natural morphism $\widetilde{\rho}_{n,n-1}:X_{n,n-1}^\F\rightarrow V_n^\F$.
 (\ref{eq:X_alpha_morphism}) shows that $\widetilde{\alpha}_n$ identifies
$\F[X_{n,n-1}]$ and $ \F[X_{n}]$. This gives rise to $\widetilde{\rho}_n:X_n^\F\rightarrow V_n^\F$
with $\widetilde{\rho}_{n,n-1}=\widetilde{\rho}_n\circ \widetilde{\alpha}_n$. The compositions
$X_n^\F\rightarrow V_n^\F\rightarrow V_n^\F/S_n$ and $X_{n}^\F\rightarrow H_n^\F\rightarrow V_n^\F/S_n$
coincide over $V_n^{0\F}/S_n$, hence coincide. Therefore we get a morphism $\iota_n:X_n^\F\rightarrow
IH_n^\F$. This morphism is finite by construction and is an isomorphism over $V_n^{0\F}\subset IH_n^\F$.

{\it Step 4}. By \cite[Lemma 3.3.1]{Haiman_n!} and the case of $n=2$,
$IH_n^\F$ is smooth over $V_n^{1\F}/S_n$. It follows that $\iota_n$ is an isomorphism
over $V_n^{1\F}/S_n$. In particular,
$X_n^\F$ is smooth outside a codimension $2$ locus. Since $X_n^\F$ is Cohen-Macaulay, it  is normal 
and $\iota_n$ is the normalization morphism. 

{\it Step 5}. Let us show that $X_n^\F$ is Gorenstein with canonical bundle
$\pi_n^*\Str_n^\F(-1)$. As in \cite[Section 6.1]{Haiman_CDM}, we use the inductive assumption -- $K_{X_{n-1}^\F}=\pi_{n-1}^*(\Str_{n-1}^\F(-1))$ -- to see that
the canonical bundle of $X_{n,n-1}^\F$ is $\pi_{n,n-1}^*(\Str_{n,n-1}^\F(0,-1))$.
Therefore $\pi_n^*(\Str_n^\F(-1))=\tilde{\alpha}_{n*}\pi_{n,n-1}^*(\Str_{n,n-1}^\F(0,-1))$
is the canonical sheaf of $X_n^\F$.

{\it Step 6}. It remains to show that $\iota_n:X_n^\F\rightarrow IH_n^\F$ is bijective.
By the inductive assumption, $\iota_{n-1}:X_{n-1}^\F\rightarrow IH_{n-1}^\F$
is bijective. Therefore the natural morphism $\iota_{n,n-1}:
X_{n,n-1}^\F\rightarrow IH_{n,n-1}^\F$ is bijective as well. The fibers
of the natural morphism $\hat{\alpha}_n:IH^\F_{n,n-1}\rightarrow IH^\F_n$ (as algebraic varieties)
coincide with the fibers of $\bar{\alpha}_n:H^\F_{n,n-1}\rightarrow U^\F_n$.
The latter fibers are all connected. So we see that the fibers of $\iota_n\circ \widetilde{\alpha}_n:
X_{n,n-1}^\F\rightarrow IH_n^\F$ are connected. It follows that the finite morphism $\iota_n$ is bijective.
\end{proof}

\subsection{Proof of main results} In this section we prove Theorems \ref{Thm:bijective_normalization} and \ref{Thm:Procesi_inductive}.

Consider the morphism $IH_n^\Ring\times_{V_n^\Ring/S_n}V_{n}^{1\Ring}/S_{n}\rightarrow H_n^\Ring$.
It is quasi-finite so canonically decomposes as $IH_n^\Ring\times_{V_n^\Ring/S_n}V_{n}^{1\Ring}/S_{n}
\rightarrow X_n^\Ring\rightarrow H_n^\Ring$ for a normal scheme $X_n^\Ring$, where the first arrow is 
an open embedding and the second arrow is finite. Thanks to Proposition
\ref{Prop:bijective_normal_p}, for $p$ sufficiently large, the base change of $X_n^\Ring$ to $\F$
is the scheme $X_n^\F$ from that proposition. Using Proposition \ref{Prop:bijective_normal_p}
again, we see that $X_n^\Ring$ is Cohen-Macaulay  after replacing $\Ring$
with a finite localization. Then we can define the bundle $\Pro_n^\Ring$ on $H_n^\Ring$
as $\pi_{n*}\Str_{X_n^\Ring}$. By Proposition \ref{Prop:bijective_normal_p}, it specializes to
$\Pro_n^\F$ over $\F$. It follows that the specialization $\Pro_n$ to $\C$ has no higher cohomology
and coincides with
$$\pi_{n*}\Str_{IH_n\times_{V_n/S_n}V_{n}^{1}/S_{n}}.$$
From the latter we see that it is a Procesi bundle. By the construction, it comes with
a morphism $\Pro_n\rightarrow \alpha_{n*}\beta_n^* \Pro_{n-1,\diamondsuit}$ that must be
an isomorphism. This proves Theorem \ref{Thm:Procesi_inductive}. Then
we can repeat the proof of the remaining parts of Proposition \ref{Prop:bijective_normal_p} over $\C$
to prove Theorem \ref{Thm:bijective_normalization}.

\end{document}